\documentclass[11pt, reqno]{article}
\usepackage{amsmath,amssymb,amscd,graphicx, color}
\usepackage[pagewise]{lineno}

\usepackage{tensor}
\newcommand{\bq}{\begin{equation}}
\newcommand{\eq}{\end{equation}}
\newcommand{\pa}{\partial}
\newcommand{\R}{{ \mathbb{R}  }}

\newcommand{\bbr}{{ \mathbb{R}  }}

\newcommand{\calN}{{ \mathcal N  }}
\newcommand{\calK}{{ \mathcal K  }}

\newcommand{\bke}[1]{\left( #1 \right)}

\newcommand{\norm}[1]{\left\Vert #1 \right\Vert}
\newcommand{\abs}[1]{\left| #1 \right|}

\newcommand{\na}{\nabla}

\newcommand {\al}{\alpha}

\newcommand{\Del}{\Delta}

\begin{document}
\bibliographystyle{plain}
%\english

% \newtheorem{thm}{Theorem}[section]
% \newtheorem{cor}[thm]{Corollary}
% \newtheorem{lem}[thm]{Lemma}
% \newtheorem{prop}[thm]{Proposition}
% \theoremstyle{definition}
% \newtheorem{defn}[thm]{Definition}
% \theoremstyle{remark}
% \newtheorem{rem}[thm]{Remark}

\newtheorem{defn}{Definition}
\newtheorem{lemma}[defn]{Lemma}
\newtheorem{proposition}{Proposition}
\newtheorem{theorem}[defn]{Theorem}
\newtheorem{cor}{Corollary}
\newtheorem{remark}{Remark}
\numberwithin{equation}{section}

\def\Xint#1{\mathchoice
   {\XXint\displaystyle\textstyle{#1}}%
   {\XXint\textstyle\scriptstyle{#1}}%
   {\XXint\scriptstyle\scriptscriptstyle{#1}}%
   {\XXint\scriptscriptstyle\scriptscriptstyle{#1}}%
   \!\int}
\def\XXint#1#2#3{{\setbox0=\hbox{$#1{#2#3}{\int}$}
     \vcenter{\hbox{$#2#3$}}\kern-.5\wd0}}
\def\ddashint{\Xint=}
\def\dashint{\Xint-}
\def\aint{\Xint\diagup}

\newenvironment{proof}{{\bf Proof.}}{\hfill\fbox{}\par\vspace{.2cm}}
\newenvironment{proof-mainthm}{{\par\noindent\bf
            Proof of Theorem \ref{thm1} }}{\hfill\fbox{}\par\vspace{.2cm}}
\newenvironment{pfthm2}{{\par\noindent\bf
            Proof of Theorem  \ref{thm2} }}{\hfill\fbox{}\par\vspace{.2cm}}
\newenvironment{pfthm3}{{\par\noindent\bf
Proof of Theorem \ref{thm3}} }{\hfill\fbox{}\par\vspace{.2cm}}
\newenvironment{pfthm4}{{\par\noindent\bf
Sketch of proof of Theorem \ref{Theorem6}.
}}{\hfill\fbox{}\par\vspace{.2cm}}
\newenvironment{pfthm5}{{\par\noindent\bf
Proof of Theorem 5. }}{\hfill\fbox{}\par\vspace{.2cm}}
\newenvironment{pflemsregular}{{\par\noindent\bf
            Proof of Lemma \ref{sregular}. }}{\hfill\fbox{}\par\vspace{.2cm}}

\title{  Global Well-posedness and Long Time Behaviors of Chemotaxis-Fluid System Modeling Coral Fertilization}
\author{Myeongju Chae, Kyungkeun Kang and Jihoon Lee}

\date{}

\maketitle
\begin{abstract}
 We consider generalized models on coral broadcast spawning phenomena involving diffusion, advection, chemotaxis, and reactions when egg and sperm densities are different.
We prove the global-in-time existence of the regular solutions of the models as well as their temporal decays in two and three dimensions. We also show that  the total masses of egg and sperm density have  positive lower bounds as time tends to infinity in three dimensions.\\  %Similar results are also obtained for the coral spawning  models coupled with the incompressible fluid equations.\\
\end{abstract}

\section{Introduction}
\setcounter{equation}{0}
In this paper, we study the interaction between reactions and chemotaxis in the mathematical model of the  broadcast spawining phenomenon. Broadcast spawning is a fertilization strategy used by many sea animals, like sea urchins and corals(see \cite{Coll1, Coll2, Lasker}). In contrast with the numerical simulations based on the turbulent eddy diffusivity, the field measurements indicate that fertilization rates are often extremely as high as 90\%(see \cite{Den-Shi, Eck} and references therein) and it seems plausible that the chemotaxis emitted by the egg gametes play an important role in these high fertilization rates.\\
\indent The simplest and most classical models of chemotaxis equations describing the collective motion of cells or bacterias have been introduced by Patlak\cite{Patlak} and Keller-Segel\cite{KS1, KS2}. The logistic source type of reaction term is also considered in many studies for the mathematical modeling of chemotaxis equations in a bounded domain with Neumann boundary conditions(see \cite{Tello-Wink, Wink1, Wink2} and references therein).\\
\indent In \cite{Kise-Ryz1, Kise-Ryz2}, Kiselev and Ryzhik initiated mathematical study on the phenomenon of broadcast spawning when males and females release sperm and egg gametes into the surrounding fluid. There is experimental evidence that eggs release a chemical that attracts sperm. The authors in \cite{Kise-Ryz1} and \cite{Kise-Ryz2} in particular have proposed the following chemotaxis model regarding the fertilization process (assuming that the densities of egg and sperm gametes are identical):
\begin{equation}\label{eq1-single}
\partial_t n +( u \cdot \nabla) n -\Delta n = \chi \nabla \cdot( n \nabla(\Delta)^{-1} n)-\epsilon n^q,\quad\mbox{ in }\,\, (x,t)\in \R^d\times (0,\, T),
\end{equation}
where $n$ is the density of egg (sperm) gametes, $u$ is the smooth divergence free sea fluid velocity, and $\chi$ denotes the positive chemotactic sensitivity constant. Also, $-\epsilon n^q$ denotes the reaction (fertilization) phenomenon. In \cite{Kise-Ryz1}, the global-in-time existence of the solution to \eqref{eq1-single} is presented under suitable conditions. Additionally, in $\R^2$, they showed that the total mass $m_0(t)=\int_{\R^2} n(x,t) dx$ approaches a positive constant whose lower bound is $C(\chi, n_0, u)$ as $t\rightarrow \infty$ when $q$ is an integer larger than 2. They also provided that $C(\chi, n_0, u)\rightarrow 0$ as $\chi \rightarrow \infty$. This implies that if the chemotactic sensitivity increases, then more eggs can be fertilized. The critical case of $d = q = 2$ was studied in \cite{Kise-Ryz2}; the total mass can go to zero with a reaction term only, but not faster than a logarithmic rate when the initial data is in the Schwartz class. If chemotaxis is present, the total mass is diminished in a power of $1/\chi$, which gives a faster decay rate than $1/\log t$ in a certain time scale. Recently, the existence and total mass behaviors have been studied in \cite{AKKL} when the chemical concentration is governed by the parabolic equation. Espejo and Suzuki \cite{Esp-Suz} considered parabolic-parabolic Keller Segel equations with reaction term coupled with Stokes equations in $\R^2$. They obtained the global-in-time existence of solution.\\
\indent Kiselev and Ryzhik\cite{Kise-Ryz1} also presented the following model of sperm and egg densities
\begin{equation}\label{system1-Kise-Ryz}\left\{
\begin{array}{ll}\partial_t s +( u \cdot \nabla) s = \kappa_1 \Delta s -(se)^{\frac{q}{2}},\quad s(x,0)=s_0(x),\\
\partial_t e +(u \cdot \nabla)e = \kappa_2 \Delta e -(se)^{\frac{q}{2}}, \quad e(x,0)=e_0(x).
\end{array}
\right.
\end{equation}
Here, $s$ and $e$ denote the densities of sperm and egg gametes. From \cite{Kise-Ryz1}, it is obtained that
if $ q > \max \{ \frac{d+2}{d},\, 2\}$, then there exists an absolute positive constant $\mu_1$ such that $ \| s(\cdot, t)\|_{L^1(\R^d)} + \| e (\cdot, t) \|_{L^1(\R^d)} \geq \mu_1 >0$ for all $t$.\\
 In this paper we consider more general mathematical models by allowing that egg density can differ from sperm density in
$\bbr^d$ ($d=2,3$) with $q=2$ considering the chemotaxis effect in the $s$ equation in \eqref{system1-Kise-Ryz}. Our first model reads as follows :
 \begin{equation}\label{Bioreaction} \left\{
\begin{array}{ll}
\partial_t e + (u \cdot \nabla)  e - \Delta e= -\epsilon (se),\\
\vspace{-3mm}\\
\partial_t s + (u \cdot \nabla) s-\Delta s =\chi \nabla \cdot (s \nabla \Delta^{-1} e)-\epsilon(se),\\
\end{array}
\right. \quad\mbox{ in }\,\, (x,t)\in \R^d\times (0,\, \infty),
\end{equation}
where
 $e \ge 0$ , $s\ge 0 $, and $u$  denote the
density of egg gametes, sperm gametes and the divergence free sea velocity of sea fluid,
respectively.  In the above, $\chi$ and $\epsilon$ are positive constants representing chemotactic sensitivity and fertilization rate, respectively.
We also assume that $u$ is in $C^{\infty}(\R^{d+1})$ and ${\rm{div}}\,\, u=0$.\\
We will obtain the apriori estimates
in section \ref{sect2}.
Initial data are given by $(e_0(x), s_0(x))$ with $e_0(x),\, s_0(x) \geq 0$.
%Thus, $\phi(x)=ax_{d}$ is one example of
%gravity force, and $\phi(x)=\phi(|x|)\rightarrow 0$ as $|x|
%\rightarrow \infty$ is an example of  centrifugal force.
\\

From now on, we denote $L^{q,p}_{t,x} = L^q(0, T; L^p(\bbr^d))$ and
$L^p_{t,x} = L^p(0, T; L^p(\bbr^d))$ with any given time $T$  in the context.
We mostly omit the spatial domain $\bbr^d$ in $L^p(\bbr^d)$ if there is no ambiguity. We also denote  a norm
\[
\| f \|_{M_n} = \int_{\R^d} (|f(x)|+|\nabla f(x)|)(1+|x|^n) dx,
\]
and Banach space $K_{m,n}$ defined by the norm $\| f\|_{K_{m,n}} = \| f\|_{M_n} + \| f\|_{H^m}$. We also denote a function space $X_{m,n}^{T}\equiv C([0, T]; K_{m,n})$ and  $X_{m,n}^{\infty}\equiv C([0, \infty); K_{m,n})$ .
Let $m_s(t)$ and $m_{e}(t)$ denote the total mass of sperm and egg gametes, respectively : $m_s(t)=\int_{\R^d} s(x,t)dx$ and $m_{e}(t)= \int_{\R^d} e(x,t) dx$.

Our first main result is the global-in-time existence of smooth solutions to \eqref{Bioreaction}. We also obtain the positive lower bound of the total mass for 3-dimensional case and the decay estimates of $\| e\|_{L^p}$ and $\| s\|_{L^p}$. 
 Compared to the case of $e$, the temporal decay of $s$ is a bit tricky, due to the presence of the chemotatic effect, i.e.  $\chi \nabla \cdot (s \nabla \Delta^{-1} e)$. It turns out that the reaction term $-\epsilon (se)$ in the egg equation, in particular in two dimenstions, plays a crucial role in controlling the chemotatic term. See the argument around \eqref{eq-s-p1}.
%(the details are given in the proof of Theorem \ref{thm2} (iii)).

\begin{theorem}\label{thm2}
Let $d=2, 3$. We assume the initial data $(e_0(x),\, s_0(x))\in K_{m,n} \times K_{m,n}$ $( m \geq [\frac{d}{2}]+1$ and $n \geq 1)$ and a given velocity field $u(x, t) \in  C([0, \infty); H^m)$ satisfies $\rm{ div }\,\,u=0$.
\begin{description}
	\item[(i)] When $d=2,3$, there exists a unique solution $(e,\, s) \in X_{m,n}^{\infty}\times X_{m,n}^{\infty}$ to the system \eqref{Bioreaction}.
	\item[(ii)] When  $d=3$, we have $m_s(t) ,\, m_{e}(t)\geq C(\chi, \epsilon, s_0, \nabla e_0)>0$. This lower bound also satisfies $C(\chi, \epsilon, s_0, \nabla e_0) \rightarrow 0$ as $\chi \rightarrow \infty$.
	\item[(iii)] When $d=2,\, 3$, we have the following temporal decay estimates
	\begin{equation}\label{decay-e}
	\| e(t) \|_{L^p(\R^d)}\leq \frac{C}{t^{\frac{d}{2}\left( 1-\frac{1}{p} \right)}}, \quad p \in (1, \infty],
	\end{equation}
	and
	\begin{equation}\label{decay-s}
	\| s(t) \|_{L^p(\R^d)}\leq \frac{C}{t^{\frac{d}{2}\left( 1-\frac{1}{p} \right)}}, \quad p \in (1, \infty).
	\end{equation}
\end{description} 
\end{theorem}
\begin{remark}
In Theorem \ref{thm2} (ii), the fact that lower bound $C(\chi, \epsilon, s_0, \nabla e_0) \rightarrow 0$ as $\chi \rightarrow \infty$ implies that if the chemotactic sensitivity is dominant, then total mass of egg or sperm density may vanish, hence perfect fertilization may occur.
\end{remark}
\indent Next, we consider the following  egg-sperm chemotaxis model coupled with the incompressible fluid equations(Navier-Stokes or Stokes equations):
 \begin{equation}\label{se-NS-intro} \left\{
\begin{array}{ll}
\partial_t e + (u \cdot \nabla)  e - \Delta e= -\epsilon (se),\\
\vspace{-3mm}\\
\partial_t s + (u \cdot \nabla) s-\Delta s =-\chi \nabla \cdot (s \nabla c)-\epsilon(se),\\
\partial_t c+(u \cdot \nabla) c -\Delta c = e,\\
\partial_t u +\kappa ( u\cdot \nabla) u -\Delta u +\nabla p =-(s+e) \nabla \phi,\quad \rm{ div }\,\,u=0,\\
\end{array}
\right. \quad\mbox{ in }\,\, (x,t)\in \R^d\times (0,\, \infty),
\end{equation}
where
 $e,\, s, \, c\ge 0 $, and $u$  denote the
density of egg gametes, sperm gametes, chemicals and the divergence free sea velocity of sea fluid governed by the fluid equations,
respectively. $\phi$ denotes potential function, which is given by gravitational force, centrifugal force, etc. We will set $\kappa=1$(Navier-Stokes equations) when $d=2$ and $\kappa=0$ (Stokes system) when $d=3$.
Chemotaxis equation coupled with the fluid equations have been considered in many studies, especially for describing the dynamics of Bacillus Subtilis in the water droplet. For recent mathematical developments in the model, please refer to \cite{CKL1, CKL3, CKLL, Esp-Suz} and references therein.\\
\indent For the system \eqref{se-NS-intro} our main aim is to establish global well-posedness of solutions.
To be more precise,  in two dimensions, we prove that unique regular solutions exist globally in time for large initial data, provided that the data are regular enough.
On the other hand, for three dimensional case, global well-posedness can be obtained under smallness condition of $L^1$-norm of intial data of $s$, i.e. 
$\norm{s_0}_{L^1}$ (more specifically, it suffices to assume that $\chi^2 \| s_0 \|_{L^1_{x}}^{2} \| \nabla \phi \|_{L^{\infty}_{x,t}}^2$ is small).
It is worth mentioning that  $L^1$-norm of $s_0$ is a {\it super-critical} qunatity in 3D under the scaling invariance \eqref{scaling-escup}
($L^{3/2}$-norm of $s_0$ is indeed scaling invariant in 3D). In this sense, our result is beyond scaling invariance but we do not know if the smallness assumption can be removed or not. Now we are ready to state our second result, where temporal decays of solutions are also shown as well.

%Direct consequences of Theorem \ref{thm1} are the following:
%\begin{enumerate}
%\item
\begin{theorem}\label{thm-exist-es-NS}
Let $d=2,\, 3$. We assume the initial data $(e_0(x),\, s_0(x),\, c_0 (x),\, u_0(x)) \in K_{m, n}\times K_{m,n} \times K_{m,n}\times H^m$ $( m \geq [\frac{d}{2}]+1$ and $n \geq 1)$ with ${\rm{ div }} \,\,u_0=0$. We also assume that $s_0,\, e_0 \in L^1(\R^d)$.
\begin{description} 
\item[(i)] When $d=2$ and $\kappa=1$, there exist unique solutions $(e,\, s,\,c,\, u) \in X_{m,n}^{\infty}\times X_{m,n}^{\infty}\times X_{m,n}^{\infty}\times C([0,\, \infty); H^m)$ to the equations \eqref{se-NS-intro}.
\item[(ii)] When $d=3$ and $\kappa=0$, assuming $\chi^2 \| s_0 \|_{L^1_{x}}^{2} \| \nabla \phi \|_{L^{\infty}_{x,t}}^2$ to be sufficiently small, there exist unique solutions $(e,\, s,\,c,\, u) \in X_{m,n}^{\infty}\times X_{m,n}^{\infty}\times X_{m,n}^{\infty}\times C([0,\, \infty); H^m)$  to the equations \eqref{se-NS-intro}.
Moreover, we have $m_e(t),\, m_s(t) \geq C(\chi, \epsilon, s_0, \nabla e_0)>0$. This lower bound  also satisfies $C(\chi, \epsilon, s_0, \nabla e_0) \rightarrow 0$ as $\chi \rightarrow \infty$.
\item[(iii)] We have the following decay estimates
\begin{equation}\label{decay-e-22}
\| e (t) \|_{L^p(\R^d)} \leq \frac{C}{t^{\left( 1-\frac{1}{p}  \right)\frac{d}{2}}},\quad\mbox{ when } 1< p \le \infty,\quad \mbox{ if } d=2,\, 3, 
\end{equation}
\begin{equation}\label{Jan-13-40}
\norm{c(t)}_{L^q}\leq
Ct^{-\frac{3}{2}(\frac{1}{3}-\frac{1}{q})},\qquad 3<q<\infty.\qquad \mbox{ if } d=3,
\end{equation}
%\begin{equation}\label{Decay-20-1}
%\norm{c(t)}_{L^q}\lesssim
%\frac{1}{(C+Ct)^{\frac{q-3}{6(q-1)}}}+C\| e_0 \|_{L^1},\qquad 3\le q\le
%\infty,\qquad \mbox{ if }\,d=3.
%\end{equation}
Furthermore, when $d=2$ and $\omega$ is the vorticity of $u$, if we assume that $\| s_0 \|_{L^1(\R^2)}+\| e_0 \|_{L^1(\R^2)}+\| \nabla c_0 \|_{L^2(\R^2)} +\| \omega_0 \|_{L^1(\R^2)} \leq \epsilon_1$, then we have
\bq \label{eq-May-25-2}
\| s(t) \|_{L^p (\R^2)} \leq \frac{C\epsilon_1}{t^{\left( 1-\frac{1}{p}\right)}},\quad \| \nabla c(t)\|_{L^{\infty}} \le \frac{C\epsilon_1}{t^{\frac12}},\quad \| \omega (t)\|_{L^{\gamma}(\R^2)} \leq \frac{C\epsilon_1}{t^{1-\frac{1}{\gamma}}},
\eq
where $1<p\leq \infty$ and $1<\gamma<2$.
\end{description}
\end{theorem}

\begin{remark}
Formally integrating both sides of \eqref{Bioreaction} $($or \eqref{se-NS-intro}$)$ over ${\mathbb{R}}^d$ and subtracting the first equation from the second equation, we deduce that
\begin{equation}\label{qwe02}
\| s \|_{L^1({\mathbb{R}}^d)} (t) - \| e \|_{L^1({\mathbb{R}}^d)}(t) = \| s_0 \|_{L^1({\mathbb{R}}^d)} - \| e_0 \|_{L^1({\mathbb{R}}^d)}, \mbox{ for all } t>0.
\end{equation} 
Hence the difference of the total mass of sperm and egg cells is conserved.\\
On the other hand, in the $2D$ case, Kiselev and Ryzhik \cite[Theorem 1.1]{Kise-Ryz2}  showed that if $\rho_0 \in {\mathcal{S}}$ (Schwartz class) and $\rho$ satisfy
\begin{equation}\label{qwe-03}
\partial_t \rho +(u\cdot \nabla) \rho -\Delta \rho = -\epsilon \rho^2,
\end{equation}
then, for any $ \sigma >0$ and $t \geq 1$, there exists a constant $C(\sigma, \rho_0) >0$ such that
\[
\| \rho (\cdot, t) \|_{L^1({\mathbb{R}}^2)} \leq \frac{C( \sigma, \rho_0)}{(1+\epsilon \log t)^{1-\sigma}}.
\]
Note that \eqref{qwe-03} corresponds to \eqref{eq1-single} when the chemotaxis is absent and $q=2$.\\
\indent
If $s(x,t) > e(x,t)$ holds true for all $(x, t) \in {\mathbb{R}}^2 \times (0, \infty)$, then $\eqref{Bioreaction}_1$ and $\eqref{se-NS-intro}_1$ are reduced to
\[
\partial_t e +(u \cdot \nabla) e -\Delta e = -\epsilon (se ) \leq - \epsilon e^2.
\]
In this case, applying Kiselev and Ryzhik's result for the solution to  the above (assuming $u$ is sufficiently regular), we obtain
\[
\| e (\cdot , t) \|_{L^1({\mathbb{R}}^2)} \leq \frac{C}{(1+ \epsilon \log t)^{1-\sigma}} \rightarrow 0 \mbox{ as } t \rightarrow \infty.
\]
Taking into account \eqref{qwe02}, we infer that, in $2D$, an egg cell can be perfectly fertilized if the initial sperm cell density is much larger than that of the egg cell.
\end{remark}
\begin{remark}
After completing this work, we are informed that Espejo and Winkler\cite{Esp-Wink} obtained classical solvability and stabilization in a chemotaxis-Navier-Stokes system modeling coral fertilization in a smooth bounded two-dimensional domain. Our result has an essential difference from their work in the asymptotic behaviour in the whole domain.
\end{remark}
\indent The rest of this paper is as follows :
In Section \ref{sect2}, we provide the proofs for the global-in-time existence of the smooth solution to \eqref{Bioreaction} and also provide the proofs of the positive lower bounds of the total mass and decay estimates. In Section \ref{sect3}, we consider the global well-posedness of the system \eqref{se-NS-intro} and provide the proof of Theorem \ref{thm-exist-es-NS} and especially consider the decay properties of the solutions to \eqref{se-NS-intro} with the small initial data.
%\item
%The authors recently showed asymptotics of solutions in
%\cite{ckl-jkms}, provided that initial data further satisfies
%\[
%\norm{n_0}_{L^1(\R^2)}+\norm{c_0}_{L^{\infty}(\R^2)}+\norm{\omega_0}_{L^1(\R^2)}<\epsilon.
%\]
%Due to Theorem \ref{thm1}, smallness condition of $c$, i.e.
%$\norm{c_0}_{L^{\infty}(\R^2)}<\epsilon$ is not necessary, since
%$\norm{c(t)}_{L^{\infty}(\R^2)}$ becomes sufficiently small in a
%finite time.
%
%{\tt...I guess that only smallness of $L^1$ of $\omega_0$ is
%required and others are not necessary...}
%\end{enumerate}}

\section{Global Well-posedness and Asymptotic Behavior of Total Mass}\label{sect2}

In this section, we provide some apriori estimates of solutions to \eqref{Bioreaction}. Also we provide the proof of global well-posedness of \eqref{Bioreaction} (Theorem \ref{thm2} (i)) and lower bound of the total mass(Theorem \ref{thm2} (ii)).
Using the standard method(contraction mapping principle), the local-in-time existence of regular solution can be shown, which % The result for local-in-time existence of a pair $(n,c)$
reads
as follows:
\begin{proposition}\label{prp1}
Let $d=2,3$ and $n$ be a positive integer and initial data $(e_0, s_0)$  as in Theorem \ref{thm2} belong to  $K_{m,n}\times K_{m,n}$ $(m > \left[ \frac{d}{2} \right]+1)$. Suppose that $u\in C^{\infty}\cap L^{\infty}(\R^{d} \times [0,\infty))$ is divergence free and any of its spatial derivatives is uniformly bounded for all $(x,t)\in \R^{d}\times (0,\infty)$. Then there exists a maximal time of
existence $T_{*}$, such that for $t<T_{*}$,  a pair of unique regular solution
$(e,s)$ of \eqref{Bioreaction} exists and satisfies
\[
(e,s) \in X^{t}_{m,n} \times X^{t}_{m,n}.
\]
\end{proposition}
The proof of the  proposition is quite standard, hence we omit it.  It can be found in \cite[Theorem 5.4]{Kise-Ryz1}.
\\

 \indent
 In this section and throughout the paper we use the  maximal $L^p-L^q$ estimates or maximal regularity estimates for  the heat equations: let 
 $ 1< p, q < \infty$.    If $v$ is the solution of the heat equation
 \[ \pa_t v- \Del v = f(x, t), \quad v(\cdot, 0) = v_0 \]
 for the given function  $f(x,t) \in L_t^qL_x^p( 0, \infty;  \bbr^d) $ and $v_0 \in W^{2,p}(\bbr^d)$,  there exist a constant  $C>0$ (see \cite{Gi-So}) such that 
 \begin{align}\label{max}
\int_0^T  \| \pa_t v \|_{L_x^p}^q dt + \int_0^T \| \Del v\|_{L_x^p}^q dt \le C\left(\| v_0\|_{W^{2,p}}^{q} + \int_0^T \| f\|_{L^p}^q dt \right).
 \end{align}
 \indent
  We often denote $(0, T)\times \bbr^d$ by $Q_T$ and $ \| v\|_{L^q_tL^p_x(0, T; \bbr^d)}$ by $L^{q, p}_{t,x} (Q_T)$.  When $p=q$, we simply write $L^p(Q_T)$. Also we denote $\sum_{|\alpha| \leq m} \| D^{\alpha} v \|_{L^q_tL^p_x(0, T; \bbr^d)}$ by $L^{q}_{t}W^{m,p}_x (Q_T)$ (or $L^q_tH^m_x$ if $p=2$).\\

  In what follows, we  derive some a priori estimates of $(e, s)$ to prove Theorem \ref{thm2}. \\
$\bullet$ ($L^1$ estimates) First, we have the following decreasing properties for the total mass
\[
\frac{d}{dt} \int_{\R^d} e(x,t) dx + \epsilon \int_{\R^d} se\,\,dx=0,
\]
and
\[
\frac{d}{dt}\int_{\R^d} s(x,t) dx + \epsilon \int_{\R^d} se\,\,dx=0.
\]
Integrating with respect to time, we have
\[
\sup_{0 \leq t \leq T} \int_{\R^d} e(x,t) dx + \epsilon \int_0^{T}\int_{\R^d} (se)dxdt\leq \| e_0 \|_{L^1},
\]
and
\[
\sup_{0 \leq t \leq T}\int_{\R^d} s(x,t) dx + \epsilon \int_0^{T}\int_{\R^d} (se) dxdt\leq \| s_0 \|_{L^1}.
\]
$\bullet$ ($L^p$-estimates) By multiplying $e^{p-1}$ both sides of $e$ equation, and integrating over ${\mathbb{R}}^d$, we obtain that\\
\[
\sup_{ 0\leq t \leq T} \int_{\R^d} e^p(x,t) dx +\frac{4(p-1)}{p} \int_0^T \| \nabla e^{p/2}\|_{L^2}^2 dt +\epsilon p \int_0^{T} \int_{\R^d} s e^{p} dxdt \leq \| e_0 \|_{L^p}^p.
\]
Moreover, as $p \rightarrow \infty$, we have $ \| e (t) \|_{L^{\infty}} \leq \| e_0 \|_{L^{\infty}}$.

For the sperm density, we have the following
\begin{equation}\label{w34}
\frac{1}{p}\frac{d}{dt} \| s(t) \|_{L^p}^p + \frac{4(p-1)}{p^2} \| \nabla s^{\frac{p}{2}} \|_{L^2}^2 + \epsilon \int_{\R^d} e s^{p} dx
 = \frac{p-1}{p}\chi \int_{\R^d} e s^{p} dx .
\end{equation}
We note that if $\epsilon \geq \chi$, then the righthand side can be absorbed to the left. Hence it is direct that
\[
s \in L^{\infty}(0, \infty; L^p)\mbox{ and } \nabla s^{p/2} \in L^2(0, \infty; L^2) \mbox{ for } p\in (1, \infty).
\]
It also holds that $s \in L^{\infty}(0, \infty; L^{\infty})$.\\
If $0<\epsilon <\chi$, then we have
\[
\frac{1}{p}\frac{d}{dt} \| s(t) \|_{L^p}^p + \frac{4(p-1)}{p^2} \| \nabla s^{\frac{p}{2}} \|_{L^2}^2 + \epsilon \int_{\R^d} e s^{p} dx
\]
\[
 = \frac{p-1}{p}\chi \int_{\R^d} e s^{p} dx \leq \frac{p-1}{p}\chi \| e\|_{L^{\infty}} \| s\|_{L^p}^p.
\]

Hence we deduce that \[
s \in L^{\infty}(0, T; L^p)\mbox{ and } \nabla s^{p/2} \in L^2(0, T; L^2) \mbox{ for  any } p\in (1, \infty) \mbox{ and } T>0.
\]
$\bullet$ ($H^1$ estimates) Next, we consider $H^1$ estimates of $s$ :\\
 By use of the  maximal regularity of heat equation, we easily deduce that
\[
\| \partial_t e \|_{L^2(Q_{T})} + \| \Delta e \|_{L^2(Q_{T})}\leq C\| e_0 \|_{H^1} +C (\| \nabla e \|_{L^2(Q_{T})} + \| se \|_{L^2(Q_{T})})< \infty.
\]
Therefore, together with ($L^p$-estimates) we obtain
\[
\partial_t e \in L^{2}_{t,x},\mbox{ and }e \in L^2(0, T; H^2).
\]
Taking $L^2$ inner product of $-\Delta s$ with $s$ equation, we find that
\[
\frac12 \frac{d}{dt} \| \nabla s \|_{L^2}^2 +\| \Delta s \|_{L^2}^2 +\epsilon \int_{\R^d} |\nabla s|^2 e dx
\]
\[
\leq \|  \nabla u \|_{L^{\infty}} \| \nabla s\|_{L^2}^2 +\epsilon  \| \nabla s\|_{L^2} \| \nabla e\|_{L^2}\| s\|_{L^{\infty}}
\]
\[
+  \chi \| s\|_{L^{\infty}} \| \nabla e\|_{L^{2}} \| \nabla s\|_{L^2} +C\chi \| \nabla s\|_{L^3}^2 \| e\|_{L^3}
\]
\[
\leq C( \| \nabla u \|_{L^{\infty}} +\| s\|_{L^{\infty}}^2+ 1)\| \nabla s\|_{L^2}^2+\delta \| \Delta s\|_{L^2}^2+ C\| \nabla e\|_{L^2}^2.
\]
In the above, $\delta$ can be chosen as a sufficiently small positive constant which can be absorbed in the lefthand side.\\
Using the Gronwall type inequality, we have for any $T>0$.
\[
\nabla s \in L^{2, \infty}_{x,t}(Q_{T})\cap H^{1}_{x}L^{2}_{t}(Q_{T}).
\]
$\bullet$ ($H^2$ estimates) For the higher norm estimates, we proceed as follows.\\
We estimate similarly with the above\\
\[
\| \partial_t \nabla e \|_{L^2(Q_{T})} + \| \nabla \Delta e \|_{L^2(Q_{T})}\leq C \| e_0 \|_{H^2} +C \| \nabla( u \cdot \nabla e)\|_{L^2(Q_{T})} +C \| \nabla (se)\|_{L^2(Q_{T})}
\]
\[
\leq C \| e_0 \|_{H^2}+C \| \nabla u \|_{L^{\infty}(Q_{T})} \| \nabla e \|_{L^2(Q_{T})} + \| u \|_{L^{\infty}(Q_{T})} \| \nabla^2 e \|_{L^2(Q_{T})}\]
\[ +C ( \| \nabla s \|_{L^2(Q_{T})} \| e\|_{L^{\infty}(Q_{T})}+\| \nabla e \|_{L^2(Q_{T})} \| s\|_{L^{\infty}(Q_{T})}) < \infty.
\]
For the estimates of solution $s$, we have
\[
\frac12 \frac{d}{dt} \| \Delta s \|_{L^2}^2 +\| \nabla \Delta s\|_{L^2}^2
\]
\[
\leq \| \nabla (u \cdot \nabla s)\|_{L^2} \| \nabla \Delta s \|_{L^2} + \| \nabla (se ) \|_{L^2} \| \nabla \Delta s \|_{L^2}
\]
\[
+C \| e\|_{L^{\infty}} \| \Delta s \|_{L^2}^2 +C\|  e\|_{L^6} \| \nabla s\|_{L^3} \|\nabla \Delta s\|_{L^2} +C \|s\|_{L^{\infty}} \| \nabla e\|_{L^2} \| \nabla \Delta s \|_{L^2}.
\]
Using Young's inequality, the righthand side high order term $\| \nabla \Delta s \|_{L^2}^2$ can be absorbed in the lefthand side. By integrating with respect to time, we find
\[
\| \Delta s\|_{L^{\infty,2}_{t,x}(Q_{T})}^2 + \| \nabla \Delta s\|_{L^{2, 2}_{t,x}(Q_{T})}^2 \leq \| \Delta s_0 \|_{L^2}^2+C \|\Delta s \|_{L^{2,2}_{t,x}(Q_{T})}^2+C\| \nabla e \|_{L^2(Q_{T})}^2< \infty.
\]
$\bullet$ ($H^3$ estimate) Finally, we can obtain the following $H^3$ estimates for $s$. \\
By the use of maximal regularity of the heat equation, we have
\[
\| \partial_t e\|_{L^2_{t}H^2_x}+ \| \Delta e \|_{L^2_{t}H^2_x}\leq  C\left(\| e_0 \|_{H^3} +\| (u\cdot \nabla)e \|_{L^2_{t}H^2_x}+\| se \|_{L^2_{t}H^2_x}     \right)
\]
\[
\leq C\left(\| e_0 \|_{H^3} +\| u \|_{L^{\infty}_{t}H^2_x} \| \nabla e \|_{L^2_t H^2_x} +\| s\|_{L^{\infty}_t H^2_x} \| e \|_{L^2_t H^2_x}\right)< \infty,
\]
and
\[
\| \partial_t e\|_{L^2_{t}H^2_x}+ \| \Delta e \|_{L^2_{t}H^2_x}\leq C\left(\| e_0 \|_{H^3} +\| u \|_{L^{\infty}_{t}(H^3_x)} \| \nabla e \|_{L^2_t (H^3_x)} +\| s\|_{L^{2}_t (H^3_x)} \| e \|_{L^{\infty}_t (H^3_x)}\right)< \infty.
\]
Similarly to the previous $H^2$ estimates, we obtain
\[
\| \nabla \Delta s \|_{L^{2, \infty}_{x,t}(Q_{T})}^2 + \| \Delta ^2 s \|_{L^{2,2}_{x,t}(Q_{T})}^2 < \infty.
\]
 We are ready to prove Theorem \ref{thm2}.
\begin{pfthm2}
{\bf{(i)}} From the previous apriori estimates, only remaining estimates are about the estimates in $M_{n}$. As in \cite[Theorem 5.4.]{Kise-Ryz1}, the only nontrivial part is that the contraction constant depends on $H^m$ norm of $(s_0, e_0)$ and not on $M_n$ norm of $(s_0, e_0)$. In a different way, we provide the following direct estimates  for any integer $k\geq 1$ inductively :
\begin{align}\label{mn}
\begin{aligned}
\frac{d}{dt} \| |x|^k e \|_{L^2}^2 +\| |x|^k \nabla e \|_{L^2}^2 + \epsilon \int se^2 |x|^{2k} dx
&
\leq C \left( \| |x|^{k-1} e \|_{L^2}^2 + \| |x|^{k-\frac12} e \|_{L^2}^2 +1\right),
\\
\frac{d}{dt} \| |x|^k s \|_{L^2}^2 +\| |x|^k \nabla s \|_{L^2}^2  + \epsilon \int es^2 |x|^{2k} dx
&
\leq C \left( \| |x|^{k-1} s \|_{L^2}^2 +\| |x|^{k-1} e \|_{L^2}^2 \| |x|^{k-\frac12} s \|_{L^2}^2 +1\right).
\end{aligned}
\end{align}
By using Young's inequality and Gronwall's inequality, we can have
for any $T>0$, \[
\| |x|^k (s,e)\|_{L^{2, \infty}_{t,x}(Q_{T})}+ \| |x|^k \nabla(s, e)\|_{L^{2}_{t,x}(Q_{T})} < \infty.
\]
Similarly, we can have $\| |x|^k \nabla (s,e)\|_{L^{2, \infty}_{t,x}(Q_{T})} < \infty.$

This together with the previous $L^1$-estimates proves for any $n>0$ and $T>0$ $ \|(s,e)\|_{K_{m,n}}< \infty$.
This completes the proof of Theorem \ref{thm2} (i). \\

\noindent{\bf{Proof of Theorem \ref{thm2} (ii)}} For this regular solution obtained in Theorem \ref{thm2} (i), we can investigate the asymptotic behaviors of the total mass $m_s(t)$ and $m_{e}(t)$, especially in ${\R}^3$.\\
First, we show $ \| e(t)\|_{L^{\infty}} \rightarrow 0$ as $t \rightarrow \infty$.\\
To be concrete, we will show that 
\[
\| e(t) \|_{L^{\infty}} \leq \frac{C}{t^{\frac{d}{2}}}.
\]
We have
\[
\frac{1}{p} \frac{d}{dt} \int_{\R^d} |e(t) |^p dx +\frac{4(p-1)}{p^2} \int_{\R^d} |\nabla e^{\frac{p}{2}} |^2 dx \leq 0.
\]
Reminding that $\| f\|_{L^2(\R^d)} \leq C \| f \|_{L^1(\R^d)}^{\frac{2}{d+2}} \| \nabla f \|_{L^2(\R^d)}^{\frac{d}{d+2}}$, we note that
\[
C \| e \|_{L^p(\R^d)}^{\frac{p(d+2)}{d}}\| e \|_{L^{\frac{p}{2}}(\R^d)}^{-\frac{2p}{d}} \leq \| \nabla e^{\frac{p}{2}} \|_{L^2(\R^d)}^2.
\]
We have
\bq\label{eq-May-25-1}
\frac{d}{dt} \int_{\R^d} |e(t)|^p dx +C \| e \|_{L^p(\R^d)}^{\frac{p(d+2)}{d}}\| e \|_{L^{\frac{p}{2}}(\R^d)}^{-\frac{2p}{d}} \leq 0.
\eq
For convenience, we denote $y_p(t) := \| e(t)\|_{L^p(\R^d)}$. We show that for sufficiently large $t>T$ and $p=2^k$ with $k=1,2,\cdots$
\[
y_{2^k} (t) \leq \frac{C_k}{t^{\left(1-\frac{1}{2^k}\right)\frac{d}{2}}}.
\]
Indeed, for $k=1$, we have
\[
\frac{d}{dt}y_2^2(t) +C y_2^{\frac{2(d+2)}{d}} \leq 0.
\]
Solving the above differential inequality, we have
\[
y_2(t) \leq Ct^{-\frac{d}{4}}.
\]
Suppose that the above is true up to $k=m-1$ with $m>1$. Then we obtain
\[
\frac{d}{dt} y_{2^m}^{2^m}(t) + \frac{C}{C_{m-1}^{\frac{2^{m+1}}{d}}} t^{2^m-2} y_{2^m}^{\frac{d+2}{d} 2^m}
\]
\[
\leq \frac{d}{dt} y_{2^m}^{2^m}(t)+C y_{2^{m-1}}^{-\frac{2}{d} 2^m} y_{2^m}^{\frac{d+2}{d} 2^m} \leq 0.
\]
Solving the above inequality, we have \eqref{decay-e-22}.
Then we have
\[
y_{p}(t) \leq \frac{C}{t^{\left(1-\frac{1}{p} \right) \frac{d}{2}}}.
\]
Letting $p \rightarrow \infty$, we have
\[
\| e(t) \|_{L^{\infty}} \leq \frac{C}{t^{\frac{d}{2}}}.
\]
$\bullet$ (Total mass behavior of $m_{s}(t)$) It is ready to prove the lower bound of mass of the sperm cell density. Consider the case that  $d=3$.
We have the differential inequality
\[
\frac{d}{dt}\int_{\R^3} s(t) dx +\frac{C}{t^{\frac32}} \int_{\R^3} s(t) dx \geq 0,\mbox{ for } t\geq t_0.
\]
Then integrating with respect to time from $t_0$ until $t$ and setting $y =\int_{\R^3} s(t) dx$, we have
\[
\frac{dy}{y} \geq -\frac{C dt}{t^{\frac32}},
\]
and thus,
\[
y(t) \geq y(t_0) \exp \left(2C\left(\frac{1}{\sqrt{t}}-\frac{1}{\sqrt{t_0}}\right)   \right).\]
Since $t\geq t_0$, we have
\[
m_{s}(t) \geq  m_{s}(t_0).
\]
$\bullet$ ($L^2$ decay estimate of $s(t)$) To prove the lower bound of the mass for the egg cell density, we should obtain $L^2$ decay estimates for the sperm cell density.\\
Similarly, we obtain
\[
\frac{d}{dt} \| s\|_{L^2(\R^3)}^2 +\| \nabla s(t) \|_{L^2(\R^3)}^2 + \epsilon \int_{\R^3} es^2 dx= \chi \int_{\R^3} e s^2 \,\, dx.
 \]
 The right hand side of the above equality can be estimated by H\"{o}lder's and Sobolev's inequality as follows :
 \[
 \chi \int_{\R^3} e s^2 (t) \,\, dx \leq \chi \| e(t) \|_{L^{\frac32}(\R^3)} \| s(t) \|_{L^6(\R^3)}^2 \leq C\chi \| e(t)\|_{L^{\frac32}(\R^3)}\| \nabla s (t)\|_{L^2(\R^3)}^2.
     \]
     Since $\| e(t) \|_{L^{\frac32}(\R^3)} \leq \frac{C}{t^{\frac12}}$, we choose $t_0$ so large that $\frac{C \chi}{t_0^{\frac12}} <\frac12$. Hence we have
\[
   \frac{d}{dt} \| s(t) \|_{L^2(\R^3)}^2 +\frac12 \| \nabla s (t)\|_{L^2(\R^3)}^2 \leq 0.
\]
We infer that $ \| s(t) \|_{L^2} \leq \frac{C}{t^{\frac34}}$.\\
$\bullet$ (Total mass behavior of $m_{e}(t)$) Finally, we deduce that 
\[
 \frac{d}{dt} \int_{\R^3} e(x,t) dx = -\epsilon \int_{\R^3} es\,\, dx \geq -\epsilon \| e(t) \|_{L^2}\| s(t)\|_{L^2} \geq -\frac{C \epsilon}{t^{\frac32}}.
\]
Similarly, we have $m_{e}(t) \geq  m_{e}(t_0)$. \\
In the above, $C$ has the order $\frac{1}{\chi}$, it implies that lower bound approaches $0$ as $\chi\rightarrow \infty$.\\

\noindent{\bf{Proof of Theorem \ref{thm2} (iii)}} We already obtained the temporal decay of $e$, that is, \eqref{decay-e}, hence we only consider the temporal decay of $s$.\\
$\bullet$ ($2D$ case) We recall that the solution  $e$ to $\eqref{Bioreaction}_2$ satisfies the equation
\begin{equation}\label{eq-e-p}
\frac{1}{p} \frac{d}{dt} \int_{\R^2} |e|^p +\frac{4(p-1)}{p^2}\int_{\R^2} |\nabla e^{\frac{p}{2}} |^2 + \epsilon \int_{\R^2} e^p s =0.
\end{equation}
Multiplying a large constant $M$ on both sides of \eqref{eq-e-p} ($M$ will be specified later), we have
\begin{equation}\label{eq-e-p-2}
\frac{M}{p} \frac{d}{dt} \int_{\R^2} |e|^p +\frac{4M(p-1)}{p^2}\int_{\R^2} |\nabla e^{\frac{p}{2}} |^2 + M\epsilon \int_{\R^2} e^p s =0.
\end{equation}
Note first that the following interpolation inequality holds (see \cite{Kise-Ryz1})
\[
\| s \|_{L^{p+1}}^{p+1} \leq C \| s\|_{L^1} \| \nabla s^{\frac{p}{2}} \|_{L^2}^2.
\]
We compute
\[
\int_{\R^2} s^pe \leq \left( \int_{\R^2} s^{\frac{p^2-1}{p}\cdot \frac{p}{p-1}}  \right)^{\frac{p-1}{p}}\left( \int_{\R^2} e^p s \right)^{\frac{1}{p}}
\]
\[
=\left(\int_{\R^2} s^{p+1}    \right)^{\frac{p-1}{p}}\left( \int_{\R^2} e^p s \right)^{\frac{1}{p}} \leq CM^{-1} \int_{\R^2} s^{p+1} +\frac{M\epsilon}{2} \int_{\R^2} e^p s
\]
\[
\leq CM^{-1} \| s\|_{L^1} \| \nabla s^{\frac{p}{2}} \|_{L^2}^2 +\frac{M \epsilon}{2} \int_{\R^2} e^p s.
\]
The solution $s(t)$ satisfies that
\begin{equation}\label{eq-s-p}
\frac{1}{p} \frac{d}{dt} \int_{\R^2} |s|^p +\frac{4(p-1)}{p^2} \int_{\R^2} |\nabla s^{\frac{p}{2}} |^2 + \epsilon \int_{\R^2} s^p e \leq CM^{-1} \| s\|_{L^1} \| \nabla s^{\frac{p}{2}} \|_{L^2}^2 +\frac{M \epsilon}{2} \int_{\R^2} e^p s.
\end{equation}
Adding \eqref{eq-e-p-2} and \eqref{eq-s-p}, we have
\[
\frac{1}{p} \frac{d}{dt} \int_{\R^2} |s|^p+\frac{M}{p} \frac{d}{dt} \int_{\R^2} |e|^p +\frac{4(p-1)}{p^2} \int_{\R^2} |\nabla s^{\frac{p}{2}} |^2 +\frac{4M(p-1)}{p^2}\int_{\R^2} |\nabla e^{\frac{p}{2}} |^2 
\]
\begin{equation}\label{eq-s-p1}
 + \epsilon \int_{\R^2} s^p e + \frac{M\epsilon}{2} \int_{\R^2} e^p s \leq CM^{-1} \| s\|_{L^1} \| \nabla s^{\frac{p}{2}} \|_{L^2}^2.
\end{equation}
Taking $M = \frac{Cp^2 \| s_0 \|_{L^1}}{2(p-1)}$, we have
\[
\frac{1}{p} \frac{d}{dt} \int_{\R^2} |s|^p+\frac{M}{p} \frac{d}{dt} \int_{\R^2} |e|^p +\frac{2(p-1)}{p^2} \int_{\R^2} |\nabla s^{\frac{p}{2}} |^2 +\frac{4M(p-1)}{p^2}\int_{\R^2} |\nabla e^{\frac{p}{2}} |^2 
\]
\[
 + \epsilon \int_{\R^2} s^p e + \frac{M\epsilon}{2} \int_{\R^2} e^p s \leq 0. \]
This gives the decay estimate
\[
\| s(t) \|_{L^p}\leq \frac{C}{t^{1-\frac{1}{p}}}\mbox{ for } p \in (1, \, \infty).
\]

$\bullet$ ($3D$ case) We estimate
\[
\int_{\R^3} s^p e \leq \| e \|_{L^{\frac32} }\| s^p \|_{L^3} = \| e \|_{L^{\frac32}}\| s^{\frac{p}{2}} \|_{L^6}^2 \leq C \| e \|_{L^{\frac32}} \| \nabla s^{\frac{p}{2}} \|_{L^2}^2 .
\]
Due to \eqref{decay-e-22}, for any given $p>1$ and sufficiently small $\delta >0$, there exists $t_0$ such that
\[
\| e(t) \|_{L^{\frac32}} < \frac{\delta}{p}\quad\mbox{ for any }\, t \geq t_0.\, 
\]
Hence we deduce that for $t \geq t_0$,
\[
\frac{1}{p} \frac{d}{dt} \int_{\R^3} |s|^p +\frac{4(p-1)}{p^2} \int_{\R^3} |\nabla s^{\frac{p}{2}} |^2 + \epsilon \int_{\R^3} s^p e= \frac{p-1}{p} \int_{\R^3} s^p e
\]
\[
 \leq C \| e\|_{L^{\frac32}} \| \nabla s^{\frac{p}{2}}\|_{L^2}^2 \leq \frac{C\delta}{p} \| \nabla s^{\frac{p}{2}}\|_{L^2}^2 .
\]
Since $\delta$ is a sufficiently small positive constant, we immediately have
\[
\frac{1}{p} \frac{d}{dt} \int_{\R^3} |s|^p +\frac{2(p-1)}{p^2} \int_{\R^3} |\nabla s^{\frac{p}{2}} |^2 + \epsilon \int_{\R^3} s^p e \leq 0.
\]
This yields that
\[
\| s(t) \|_{L^p} \leq \frac{C}{t^{\frac32\left( 1-\frac{1}{p}  \right)}} \mbox{ for } p \in (1, \, \infty).
\]
This completes the proof of Theorem \ref{thm2}.
 \end{pfthm2}

\begin{remark}
In two dimensions, we have $\| e(t) \|_{L^{\infty}} \leq \frac{C}{t}$. Then via similar computations as above, we obtain
\[
m_{s}(t)\geq \left( \frac{t_0}{t}  \right)^{C} m_{s}(t_0)\mbox{ for } t \geq t_0.
\]
Hence, in two dimensions, we can not obtain the positive lower bound of the total mass via same method in three dimensions and leave as an open problem.  
\end{remark}
        
\section{Global well-posednes  for the model \eqref{se-NS-intro}}\label{sect3}
In this section, we prove the global well-posedness of solutions to the system \eqref{se-NS-intro}.
 \begin{equation*} \left\{
\begin{array}{ll}
\partial_t e + (u \cdot \nabla)  e - \Delta e= -\epsilon (se),\\
\vspace{-3mm}\\
\partial_t s + (u \cdot \nabla) s-\Delta s =-\chi \nabla \cdot (s \nabla c)-\epsilon(se),\\
\partial_t c+(u \cdot \nabla) c -\Delta c = e,\\
\partial_t u +\kappa ( u\cdot \nabla) u -\Delta u +\nabla p =-(s+e) \nabla \phi,\quad \rm{ div }\,\,u=0,\\
\end{array}
\right. \quad\mbox{ in }\,\, (x,t)\in \R^d\times (0,\, \infty),
\end{equation*}
 We will set $\kappa=1$ (Navier-Stokes system) when $d=2$ and $\kappa=0$ (Stokes system) when $d=3$ as mentioned in Section $1$.
 \\
 \indent Note that the solution $(e,\, s,\, c,\, u,\, p)$ satisfies the scaling invariant property if $\phi$ has the following scaling property : $\phi(x,t)=\phi^{\lambda}(x,t):= \phi(\lambda x, \, \lambda^2 t)$.
That is,
\[
(e^{\lambda}(x,t),\, s^{\lambda}(x,t),\, c^{\lambda}(x,t),\, u^{\lambda}(x,t),\, p^{\lambda}(x,t))
\]
\begin{equation}\label{scaling-escup}
=(\lambda^2 e(\lambda x, \lambda^2 t), \,\lambda^2 s(\lambda x, \lambda^2 t), \, c(\lambda x, \lambda^2 t), \,\lambda u(\lambda x, \lambda^2 t), \,\lambda^2 p(\lambda x, \lambda^2 t))
\end{equation}
is also a solution to \eqref{se-NS-intro} if  $(e,\, s,\, c,\, u,\, p)$ is a solution.\\
\indent
The local-in-time existence of  the solutions to \eqref{se-NS-intro} is  obtained by the contraction as for  Proposition \ref{prp1}. Hence we omit its proof. 
Moreover similar estimates as \eqref{mn}  for the $M_n $ norm of $(e, s, c)(\cdot, T)$ are bounded by 
$\| (e_0, s_0, c_0)\|_{M_n}$ and $\| (e, s, c, u)\|_{C(0, T; H^m)}$. Thus the local solution is extended if 
$\| (e, s, c, u))\|_{C(0, T; H^m)}$ is uniformly bounded. \\
\indent
Let $T^*$ be the maximal time of existence of the local solution
and $T$ be any time until $T^*$. In what follows we shall establish a priori 
estimates for $\| (e, s, c, u)\|_{C(0, T; H^m)}$ where $m= [ \frac d2]+1$. All integrations are over $Q_T$. We often omit 
$Q_T$ in $L^q_tL^p_x(Q_T)$.

$L^1$ estimates of $ e, s, c$ and $L^p$ estimates of $ e, c$ are immediate. We have
\[
\int_{\R^d} e(T)\, dx+ \epsilon \int_0^{T}\int_{\R^d} (se) dx  dt= \int_{\R^d} e_0 \,dx,
\]
\[
\int_{\R^d} s(T) \,dx+ \epsilon\int_0^T \int_{\R^d} (se) dx dt  =\int_{\R^d} s_0 \,dx,
\]
\[
\int_{\R^d} c(T) \,dx=\int_{\R^d} c_0 \,dx +\int_0^T\int_{\R^d} e(x,t)\, dxdt.
\]
For $1<p<\infty$ we have
\[
\frac{1}{p}\frac{d}{dt} \| e(t) \|_{L^p}^p +\frac{4(p-1)}{p^2}\|\nabla  e^{\frac{p}{2}}\|_{L^2}^2 + \epsilon \int_{\R^d}(se^{p}) (x,t) dx =0,
\]
\[
\frac{1}{p}\frac{d}{dt} \| c(t) \|_{L^p}^p +\frac{4(p-1)}{p^2}\|\nabla  c^{\frac{p}{2}}\|_{L^2}^2  =\int_{\R^d}( ec^{p-1}) (x,t) dx \leq \| e\|_{L^p} \| c\|_{L^p}^{p-1} .
\]
Hence 
it holds that 
 \begin{align*}
 \begin{aligned}
 \| e\|_{L^{\infty}(0, T^{*}; L^p)}^p + \| \nabla e^{\frac{p}{2}}\|_{L^2(0, T^{*}; L^2)}^2 & \leq C\| e_0 \|_{L^p}^p,\\
\| c\|_{L^{\infty}(0, T^{*}; L^p)}^p + \| \nabla c^{\frac{p}{2}}\|_{L^2(0, T^{*}; L^2)}^2& \leq C\left(\| c_0 \|_{L^p}+ \int_0^T \| e\|_{p} dt \right)^p \leq  C(\| c_0\|_{L^p} + T \|e_0\|_{L^p})^p.
\end{aligned}
\end{align*}
To obtain other $L^p$ and higher norm estimates we first consider the estimates of $u$;
\begin{equation}\label{NS-1}
\left\{\begin{array}{cl}
\partial_t u+( u \cdot \nabla) u -\Delta u +\nabla p =-(s+e) \nabla
\phi, \quad \nabla \cdot u=0\qquad
&\mbox{in } \R^2\times (0, T)\\
u(x,\, 0)=u_0(x) \qquad &\mbox{ in } \R^2.
\end{array}\right.
\end{equation}
%and
%\begin{equation}\label{NS-2}
%\begin{array}{cl}
%\partial_t u-\Delta u +\nabla p =-(s+e) \nabla
%\phi, \quad \nabla \cdot u=0\qquad
%&\mbox{in } \R^3\times (0, T)\\
%u(x,\, 0)=u_0(x) \qquad &\mbox{ in } \R^3.
%\end{array}
%\end{equation}
 Let us denote the Stokes operator by $G_t$. Namely $G_t \ast u_0$ is the solution of the free Stokes equations
 ($f=0$)
\[ \partial_t u-\Delta u +\nabla p =f , \quad \nabla \cdot u = 0 \]
with initial data $u_0$. It is well known that $G_t$ satisfies that (see e.g. \cite{Gi-So})
\begin{align}\label{stokes2}
\| G_t * f\|_{L^p}  \le  C t^{\frac 1p -1} \| f\|_{L^1} , \quad
\| \na G_t * f\|_{L^p}  \le  C t^{\frac 1p -\frac 32 } \| f\|_{L^1} \quad 1 \le p \le \infty
\end{align}
in two dimensions.
%\begin{align}\label{stokes3}
%\| G_t * f\|_{L^p}  \le  C t^{\frac32\left(\frac 1p -1\right)} \| f\|_{L^1} , \quad
%\| \na G_t * f\|_{L^p}  \le  C t^{\frac32\left(\frac 1p -1\right) -\frac12 } \| f\|_{L^1} \quad 1\le p \le \infty
%\end{align}
%in three dimensions.
For the inhomogeneous Stokes equations the following maximal regularity estimate is known \cite{Gi-So}; 
\begin{align}\label{stokesmax}
\int_0^T  \| \pa_t u \|_{L_x^p}^q dt + \int_0^T \| \Del u\|_{L_x^p}^q dt + \int_0^T \|\nabla p\|_{L^p} ^q dt 
 \le C\left(\| u_0\|_{W^{2,p}}^{q} + \int_0^T \| f\|_{L^p}^q dt \right)
\end{align}
for $ 1<p, q <\infty$.
\begin{lemma}\label{uL4}
Let $d=2$ and $s, e, u$ be the local solution of \eqref{se-NS-intro} in $K_{m,n}$.
%{\color{red}{In two dimensional spatial domain, }} 
The solution $u$ to \eqref{NS-1} belongs to $L^{\infty}(0, T ; L^2)
\cap L^2(0, T; W^{1,q})\cap L^4(0, T ; L^4)$ for any $q\in
[1,2)$. %and {\color{red}in three dimensional spatial domain,} the solution $u$ to \eqref{NS-2} belongs to $L^{\infty} (0, T; L^p) \cap L^{\infty}(0, T;W^{1,q})$ for any $ p \in [1, 3)$ and $q \in [1, \frac32)$.
\end{lemma}
%\marginpar{$u \in L^4_t L^4_x$ used only later?}
\begin{proof}
We remind that total masses of $s$ and $e$ are preserved.
Thus, $s\nabla \phi,\, e \nabla \phi$ belong to $L^{\infty}([0,T_0);L^1(\R^2))$,
since $\phi$ is assumed to satisfies $\norm{\nabla^l
\phi}_{L^{\infty}}<\infty$ for $1\le\abs{l}\le m$.\\
%$\bullet$ Case I ($d=2$) 
Let $Q:=(0, T)\times \bbr^2 $.
 We decompose the solution $u$ to the
equations \eqref{NS-1} to  $v + w$ in $Q$, where $v$ satisfies the Stokes system:
\begin{equation}\label{Stokes-1}\left\{
\begin{array}{cl}
\partial_t v -\Delta v +\nabla p_1 =-(s+e) \nabla
\phi, \qquad {\rm{div}}\,\, v=0\qquad &\mbox{in } Q,\\
v(x,\, 0)=u_0(x) \qquad &\mbox{ in } \R^2,
\end{array}\right.
\end{equation}
and $w$ satisfies
 a perturbed homogeneous Navier-Stokes equations with zero initial data:
 \begin{equation}\label{perturbed}\left\{
\begin{array}{cl}
\partial_t w -\Delta w +\nabla p_2 =-((v+w)\cdot\nabla)v -((v+w)\cdot \nabla)w,\quad
{\rm{div}}\,\, w=0,&\mbox{in } Q,\\
w(x,\, 0)=0\qquad& \mbox{ in } \R^2.
\end{array}\right.
\end{equation}
 For convenience, we denote $f:=-(s+e)\nabla\phi$.
 By \eqref{stokes2}  we have
\[
\| v\|_{L^{\infty, p}_{t,x}(Q_T)}\leq C\| u_0 \|_{L^p}+C
\left(\int_0^{T_0} t^{\frac{1}{p}-1} dt\right) \| f \|_{L^{
\infty,1}_{t,x}(Q_T)}< \infty.
\]
for any $p \in [1,\, \infty)$.
Similarly,  we have
\[
\norm{\nabla v}_{L^{\infty, q}_{t,x}(Q_T)}\leq C\| \nabla u_0
\|_{L^q}+C \left(\int_0^{T_0} t^{\frac{1}{q}-\frac{3}{2}} dt\right)
\| f \|_{L^{\infty,1}_{t,x}(Q_T)}< \infty
\]
for any $ q \in [1,2)$.
Note that  $\| f \|_{L^{\infty,1}_{t,x}(Q_T)}\leq
C(\norm{s_0}_{L^1(\R^2)}+\norm{e_0}_{L^1(\R^2)})$. Summing up, we obtain
\begin{equation}\label{estimate-v}
\| v\|_{L^{ \infty,p}_{t,x}(Q_T)}+\norm{\nabla v}_{L^{
\infty,q}_{t,x}(Q_T)}\le C=C(T_0),\qquad p \in [1,\, \infty),\quad q\in
[1,\, 2),
\end{equation}
which yields that
\[ v \in  L^{\infty}(0, T_0; L^2)\cap L^2(0, T_0; W^{1, q} )\cap L^4(0, T_0; L^4) \quad
q\in [1, 2).\]
%
%In the following, we use the case $p=4$ :
%\[ \| v\|_{L^{4, \infty}_{x, t} }< \infty.
%\]
For the Navier-Stokes part $w$, we estimate
\[
\frac12 \frac{d}{dt} \| w \|_{L^2}^2 +\| \nabla w \|_{L^2}^2\leq
\left| \int_{\R^2} ((v+w) \cdot \nabla) w \cdot v dx \right|
%+ \left|\int_{\R^2} (w \cdot \nabla)w \cdot v dx   \right|
\]
\[
\leq \| v\|_{L^4}^2 \| \nabla w \|_{L^2} + \| w\|_{L^4} \| \nabla w \|_{L^2} \| v\|_{L^4}
\]
\[ \leq   \| v\|_{L^4}^2 \| \nabla w \|_{L^2} + C \|w\|_{L^2}^{\frac 12} \| \na w\|_{L^2}^{\frac 32} \|v\|_{L^4} \]
%\[
%\leq \| v\|_{L^4}^2 \| \nabla w \|_{L^2} + C \| w \|_{L^2}^{\frac12}
%\| \nabla w \|_{L^2}^{\frac32} \| v\|_{L^4}
%\]
\[
\leq \frac12 \| \nabla w\|_{L^2}^2 +C \| v\|_{L^4}^4(\| w\|_{L^2}^2 +1)
\]
which implies
%\[
%\| w(t)\|_{L^2}^2 \leq \left( \| u_0 \|_{L^2}^2 +C \int_0^t  \| v\|_{L^4}^4 ds \right) \exp \left( C \int_0^t \| v\|_{L^4}^4 ds  \right).
%\]
\[ w\in L^{\infty,2}_{t,x}(Q_T) \cap L^2(0, T; {H}^1_0)\]
by the Gronwall's inequality. 
%which also implies that $w\in L^{4}(Q_T)$. Since $u=v+w$ satisfies the
%equations of \eqref{NS-1}, combining \eqref{estimate-v}, it is
%straightforward that $u\in L^{\infty}(0, T_0 ; L^2) \cap L^{4}(Q_T)$.
It
remains to show that 
$ w\in \bigcap_{1 \leq q <2} L^2(0, T; W^{1,q})$. 
% which is enough
%to prove $w\in \bigcap_{1 \leq q <2} L^2(0, T; W^{1,q})$, because of
%\eqref{estimate-v}.
Using the  Stokes operator, we write  $w$ as 
\[
\nabla w(x,\, t) =-\int_0^t \nabla G_{t-s} *\bke{(v\cdot\nabla)
v+(v\cdot\nabla) w+(w\cdot\nabla) v+(w\cdot\nabla) w}(s) ds.
\]
\[
=-\int_0^t \nabla G_{t-s} *((v\cdot\nabla) v(s)) ds-\int_0^t \nabla
G_{t-s}
*((v\cdot\nabla) w(s)) ds
\]
\[
-\int_0^t \nabla G_{t-s} *((w\cdot\nabla) v(s)) ds-\int_0^t \nabla
G_{t-s}
*((w\cdot\nabla) w(s)) ds:=I_1+I_2+I_3+I_4.
\]
What it follows,  we separately compute $I_i$, $i=1,2,3,4$.
\[
\norm{I_1(t)}_{L^q}\leq \int_0^t \| \nabla G_{t-s} * ((v\nabla)
v)(s)\|_{L^q} ds \leq C \int_0^t (t-s)^{\frac{1}{q}-\frac{3}{2}}\|
v\nabla v \|_{L^1(\R^2)}(s) ds
\]
\[
\leq C \int_0^t (t-s)^{\frac{1}{q}-\frac{3}{2}}\norm{
v}_{L^{q'}}(s)\norm{\nabla v}_{L^q}(s) ds\leq C(T_0)\norm{
v}_{L^{\infty, q'}_{t,x}(Q_T)}\norm{ \nabla
v}_{L^{\infty, q}_{t,x}(Q_T)}.
\]
Similarly,
\[
\norm{I_2(t)}_{L^q}
%\leq \int_0^t \| \nabla G_{t-s} * ((v\nabla) w)(s)\|_{L^q} ds
\leq C \int_0^t (t-s)^{\frac{1}{q}-\frac{3}{2}}\| v\nabla w
\|_{L^1(\R^2)}(s) ds
\]
\[
\leq C\norm{ v}_{L^{\infty,2}_{t,x}(Q_T)}\int_0^t
(t-s)^{\frac{1}{q}-\frac{3}{2}}\| \nabla w \|_{L^2(\R^2)}(s) ds.
\]
Therefore, using the convolution inequality, we have
\[
\norm{I_2}_{L^{2,q}_{t,x}(Q_T)}\leq C(T_0)\norm{
v}_{L^{\infty,2}_{t,x}(Q_T)}\norm{\nabla w}_{L^2_{t,x}(Q_T)}.
\]
For $I_3$, using $w\in L^4(Q_T)$, we observe that
\[
\norm{I_3(t)}_{L^q}
%\leq \int_0^t \| \nabla G_{t-s} * ((v\nabla) w)(s)\|_{L^q} ds
\leq C \int_0^t (t-s)^{\frac{1}{q}-\frac{3}{2}}\| w\nabla v
\|_{L^1(\R^2)}(s) ds
\]
\[
\leq C\int_0^t (t-s)^{\frac{1}{q}-\frac{3}{2}}\|  w
\|_{L^4(\R^2)}(s)\| \nabla v \|_{L^{\frac{4}{3}}(\R^2)}(s) ds
\]
\[
\leq C\| \nabla v \|_{L^{\infty, \frac{4}{3}}_{t,x}(Q_T)}\int_0^t
(t-s)^{\frac{1}{q}-\frac{3}{2}}\|  w \|_{L^4(\R^2)}(s)(s) ds.
\]
Using the convolution inequality again, we obtain
\[
\norm{I_3}_{L^{2,q}_{t,x}(Q_T)}\leq C(T)\norm{ \nabla
v}_{L^{\infty, \frac{4}{3}}_{t,x}(Q_T)}\norm{ w}_{L^4_{t,x}(Q_T)}.
\]
Finally, we compute
\[
\norm{I_4(t)}_{L^q} \leq C \int_0^t
(t-s)^{\frac{1}{q}-\frac{3}{2}}\| w\nabla w \|_{L^1(\R^2)}(s) ds
\]
\[
\leq C \int_0^t (t-s)^{\frac{1}{q}-\frac{3}{2}}\norm{
w}_{L^{2}}(s)\norm{\nabla w}_{L^2}(s) ds
\]
\[
\leq C\norm{ w}_{L^{\infty,2}_{t,x}(Q_T)}\int_0^t
(t-s)^{\frac{1}{q}-\frac{3}{2}}\| \nabla w \|_{L^2(\R^2)}(s) ds.
\]
Similarly we get
\[
\norm{I_4}_{L^{2,q}_{t,x}(Q_T)}\leq C(T)\norm{
w}_{L^{\infty,2}_{t,x}(Q_T)}\norm{\nabla w}_{L^2(Q_T)}.
\]
Summing up estimates, we obtain that $\nabla w\in \bigcap_{ 1\leq q <2} L^2(0, T;
L^{q}(\R^2))$.
%\\
%$\bullet$ Case II ($d=3$) Let $G_t$ be the 3 dimensional Stokes operator. $u$ can be represented as
%\[
%u(x,\, t) =G_{t}* u_0 +\int_0^t G_{t-s} *f(s ) ds.
%\]
 %We obtain
%\[
%\int_0^t \| G_{t-s} * f(s)\|_{L^p} ds \leq C \int_0^t (t-s)^{\frac{1}{p}-1}\| f \|_{L^1(\R^2)}(s) ds.
%\]
%By \eqref{stokes3} we have
%\[
%\| u\|_{L^{\infty,p}_{t,x}(Q_T)}\leq C\| u_0 \|_{L^p}+C
%\left(\int_0^{T_0} t^{\frac32(\frac{1}{p}-1)} dt\right) \| f \|_{L^{
%\infty,1}_{t,x}(Q_T)}< \infty
%\] for any $p \in [1,\, 3)$
%Similarly,  we have
%\[
%\norm{\nabla u}_{L^{\infty,q}_{t,x}(Q_T)}\leq C\| \nabla u_0
%\|_{L^q}+C \left(\int_0^{T_0} t^{\frac{3}{2q}-{2}} dt\right)
%\| f \|_{L^{\infty,1}_{t,x}(Q_T)}< \infty
%\]  for any $1\le q<\frac32$.
%Therefore, we obtain
%\begin{equation}\label{estimate-u}
%\| u\|_{L^{\infty, p}_{t,x}(Q_T)}+\norm{\nabla u}_{L^{
%\infty, q}_{t,x}(Q_T)}\le C=C(T_0),\qquad p \in [1,\, 3),\quad q\in
%[1,\, 3/2),
%\end{equation}
This completes the proof.
\end{proof}
\begin{remark}
If we consider 
\begin{equation}\label{NS-2}
\begin{array}{cl}
\partial_t u-\Delta u +\nabla p =-(s+e) \nabla
\phi, \quad \nabla \cdot u=0\qquad
&\mbox{in } \R^3\times (0, T)\\
u(x,\, 0)=u_0(x) \qquad &\mbox{ in } \R^3.
\end{array}
\end{equation}	
then similarly to Lemma \ref{uL4}, we can prove that the solution $u$ to \eqref{NS-2} belongs to $L^{\infty} (0, T; L^p) \cap L^{\infty}(0, T;W^{1,q})$ for any $ p \in [1, 3)$ and $q \in [1, \frac32)$.
\end{remark}
We proceed other $L^p$ and higher order estimates to conclude the global well-posedness part of Theorem $2$. 
We treat spatial two and three dimensional cases separately.\\

{\bf{Proof of Theorem \ref{thm-exist-es-NS} {\bf{(i)}} ($d=2$)}}  If we consider the equation
\[
\partial_t c -\Delta c =-\nabla \cdot (uc) +e,
\]
then by the maximal regularity of the heat equation \eqref{max}  we obtain
\begin{equation}\label{2dc}
\| \nabla c \|_{L^{4}_{t,x}} \leq C(\|c\|_{L^{\infty}_{t,x}}\| u \|_{L^{4}_{t,x}} +\| e\|_{L^4_tL^{ 3/2}_x})+\| \nabla c_0 \|_{L^{4}_x}< \infty,
\end{equation}
where the second inequality is due to  $L^p$ esitmates of $c, e$ and Lemma \ref{uL4}. Multiplying both sides of the equation of $s$ by $s^{p-1}$ and integrating over $\R^2$, we deduce that
\begin{align}\label{2ds}
\begin{aligned}
\frac{1}{p}\frac{d}{dt} \| s\|_{L^p}^p + & \frac{4(p-1)}{p^2} \| \nabla s^{\frac{p}{2}}\|_{L^2}^2  \leq \frac{2(p-1)}{p} \chi \left| \int_{\R^2} s^{\frac{p}{2}} \nabla c \cdot \nabla s^{\frac{p}{2}} \right| \\
&
\leq \frac{2(p-1)}{p} \chi \| s^{\frac{p}{2}} \|_{L^2}^{\frac12} \| \nabla c\|_{L^4} \| \nabla s^{\frac{p}{2}} \|_{L^2}^{\frac32}
\\
&
\leq C \chi^4 \| \nabla c\|_{L^4_x}^4 \| s\|_{L^p}^p +\frac{2(p-1)}{p^2} \| \nabla s^{\frac{p}{2}}\|_{L^2}^2.
\end{aligned}
\end{align}
Hence we have
\begin{equation}\label{s_lp}
\sup_{ 0 \leq t \leq T} \| s(t)\|_{L^p}^p \leq \| s_0 \|_{L^p}^p \exp \left( C \chi^4 \| \nabla c\|_{L^4_{t,x}}^4  \right)< \infty\mbox{ for all } p \in [2,\, \infty).
\end{equation}
Therefore, $ s \in L^{\infty}(0,\, T; L^p)$ and $\nabla s^{\frac{p}{2}} \in L^2(0,\, T; L^2)$ for all $ p \in [2,\, \infty)$.\\
On the other hand, we have
\[
\frac12 \frac{d}{dt} \| u \|_{L^2}^2 +\| \nabla u \|_{L^2}^2 \leq C( \| s\|_{L^2} +\| e\|_{L^2}) \| u \|_{L^2}.
\]
It gives us that $u \in L^{\infty}(0,\, T; L^2)$ and $\nabla u \in L^2(0,\, T; L^2)$.
\[
\frac12 \frac{d}{dt} \| \nabla u \|_{L^2}^2 + \| \Delta u \|_{L^2}^2 \leq C( \| s\|_{L^2}^2 +\| e\|_{L^2}^2 + \| \nabla u \|_{L^2}^4) + \frac12 \| \Delta u \|_{L^2}^2.
\]
Therefore, we also have $\nabla u \in L^{\infty}(0,\, T; L^2)$ and $\Delta u \in L^2(0,\, T; L^2)$, that is 
\begin{equation}\label{u_0th}
u \in L^{\infty}(0,\, T; L^2), \quad \nabla u \in L^{\infty}(0,\, T; L^2), \quad \Delta u \in L^2(0,\, T; L^2).
\end{equation}
In general the maximal regularity of the heat equation and the $L^p$ estimates of $c, e$ yield that 
\begin{align}\label{c_0th}
\begin{aligned}
\| \nabla c\|_{L^p_{t,x}} &\leq C( \| u \|_{L^p_{t,x}} +1) < \infty, \\
\|  \Delta c\|_{L^p_{t,x}} & \leq C(\| u \cdot \na c\|_{L^p_{t,x}} + 1)  \le 
C( \| u \|_{L^{q}_{t,x}}\| \na c\|_{L^{pq/(q-p)}_{t,x}} +1) < \infty
\end{aligned}
\end{align}
 for all $ p \in [2,\, \infty)$ and $q>p$. We can replace $c$ with $e$ in the above.
Applying the maximal regularity of the heat equation to  $s$ equation together with the previous estimates,
 we have
\begin{align}\label{s_0th}
\| \nabla s \|_{L^p_{t,x}} ,\, \| \Delta s \|_{L^p_{t,x}} < \infty\mbox{ for all } p \in [2, \, \infty).
\end{align}
Then by the bootstraping argument, we complete the proof of the Case I. 
{Indeed \eqref{u_0th} and \eqref{s_0th} yields $L^p$ estimate for $\na c, \na e$.  Then $L^p$ estimate of $\na s$ follows from  the boundedness of $\| \Del c\|_{L^p_{t,x}}$ in \eqref{c_0th} 
as is obtained $\|s\|_{L^p}$ in  \eqref{s_lp}. Those $L^p$ estimates are used to yield 
$\na u \in L^{\infty}(0,\, T; L^2),  \nabla^2 u \in L^{\infty}(0,\, T; L^2),  \na^3 u \in L^2(0,\, T; L^2)$, which closes the $H^1$ estimate of $ e,c,s,u$.  Maximal regularity estimates for $\na c, \na e, \na s$ prove 
the boundedness of  $ \| \na c, \na e, \na s \|_{L^p_t W^{2, p}_x} $ for all $p \in [2, \infty)$, which corresponds to one more derivative version of \eqref{c_0th} and \eqref{s_0th}.  The $H^2$ estimates can be similarly done. }\\

{\bf{(ii)}} ($d=3$) We assume that $\chi^2 \| \nabla \phi\|_{L^{\infty}}^2 \| s_0 \|_{L^1}^2$ is sufficiently small. Note that $\chi^2 \| \nabla \phi\|_{L^{\infty}}^2 \| s_0 \|_{L^1}^2$ is scaling invariant quantity.  \\
\indent 
{In the three dimensional case the regularity of $u$ obtained in Lemma \ref{uL4} is not enough to prove \eqref{2dc} and \eqref{2ds} as is in two dimensions.  We need to prove an entropy type inequality for $s$ \eqref{May-2nd-5} for three dimensions.}
Taking $\log s$ as a test function for the equation $\eqref{se-NS-intro}_2$, we have
\[
\frac{d}{dt} \int_{\R^3}s \log s dx + \int_{\R^3} | \nabla \sqrt{s}|^2 dx= \int_{\R^3} \chi \nabla s \cdot \nabla c dx - \epsilon\int_{\R^3} se(1+ \log s) dx
\]
We estimate
\[
\left|\int_{\R^3} \chi \nabla s \nabla c \right| = 2 \chi \left|\int_{\R^3} \sqrt{s} \nabla \sqrt{s} \cdot \nabla c dx \right|\leq C \chi \| \nabla \sqrt{s} \|_{L^2} \| \sqrt{s} \|_{L^{\frac{30}{11}}} \| \nabla c \|_{L^{\frac{15}{2}}} \]
\[
\leq \frac12 \| \nabla \sqrt{s} \|_{L^2}^2 + C\chi^2 \| s \|_{L^{\frac{15}{11}}} \| \nabla c\|_{L^{\frac{15}{2}}}^2.
\]
Also we note that 
\[
-\int_{\{ x: s(x) \leq 1\}} se \log s dx \leq C \int_{\R^3} e dx.
\]
Hence we deduce that
\[
\frac{d}{dt} \int_{\R^3}s \log s dx + \frac12\int_{\R^3} | \nabla \sqrt{s}|^2 dx \leq C\chi^2 \| s \|_{L^{\frac{15}{11}}} \| \nabla c\|_{L^{\frac{15}{2}}}^2
\]
\[
\leq C\left( \chi^2\| s \|_{L^1}^{\frac35} \| s \|_{L^3}^{\frac25} \| \nabla c \|_{L^{\frac{15}{2}}}^2+\| e_0 \|_{L^1}\right) \leq C\left(\chi^2\| s_0 \|_{L^1}^{\frac35}\| s \|_{L^3}^{\frac25} \| \nabla c \|_{L^{\frac{15}{2}}}^2+\| e_0 \|_{L^1}\right) .
\]
Integrating in time gives us that
\[
\int_{\R^3}  s(t) \log s(t) dx - \int_{\R^3}s_0 \log s_0 dx + \frac12\int_0^t\int_{\R^3} | \nabla \sqrt{s}|^2 dxds \]
\[\leq C\left[\chi^2 \| s_0 \|_{L^1}^{\frac35}\left( \int_0^t \| s \|_{L^3} ds \right)^{\frac25}\left(\int_0^t \| \nabla c \|_{L^{\frac{15}{2}}}^{\frac{10}{3}}ds    \right)^{\frac35}+t\| e_0 \|_{L^1} \right].
\]
%We note that
%\[
%%\| s^{\lambda}_0 \|_{L^1(\R^3)} = \frac{1}{\lambda} \| s_0 \|_{L^1(\R^3)}.
%\]
Considering the equation of $c$
\[
c_t -\Delta c =-\nabla \cdot (u c) +e,
\]
and by the fact that $ e \in L^{\infty}_{t,x}$, we have
\[
\| \nabla c \|_{L^{\frac{10}{3},\frac{15}{2}}_{t,x}}\leq C (\|u c \|_{L^{\frac{10}{3},\frac{15}{2}}_{t,x}}+\| e \|_{L^{\frac{10}{3},\frac{15}{2}}_{t,x}})+\| \nabla c_0 \|_{L^{\frac{15}{2}}_{x}}
\]
\[
\leq C ( \| u \|_{L^{\frac{10}{3},\frac{15}{2}}_{t,x}}+1) \leq  C ( \|\nabla^2 u \|_{L^{ \frac{10}{3},\frac{5}{4}}_{t,x}}+1) \leq C( \| s \|_{L^{\frac{10}{3},\frac{5}{4}}_{t,x}}\| \nabla \phi \|_{L^{\infty}_{t,x}}+1),
\]
where the last inequality is from \eqref{stokesmax}.
Since we have $\| s \|_{L_x^{\frac{5}{4}}} \leq C \| s \|_{L_x^1}^{\frac{7}{10}} \| s \|_{L_x^{3}}^{\frac{3}{10}}$, we deduce that
 \[
\int_{\R^3}  s \log s dx - \int_{\R^3}s_0 \log s_0 dx + \frac12\int_0^t\int_{\R^3} | \nabla \sqrt{s}|^2 dxds\]
\[ \leq C\chi^2 \| s \|_{L^{1,3}_{t,x}}^{\frac25} \left( \| s \|_{L^{1,3}_{t,x}}^{\frac35}\| \nabla \phi  \|_{L^{\infty}_{t,x}}^2\| s_0 \|_{L^1_x}^{2}+1\right)+ Ct \| e_0 \|_{L^1}
\]
\[
\leq \left[C_{*}\chi^2\| \nabla \phi  \|_{L^{\infty}_{t,x}}^2\| s_0 \|_{L^1_x}^{2}+\frac18\right]\| \nabla \sqrt{s}\|_{L_{t,x}^2}^2+Ct.
\]
Therefore, from the assumption that $C_{*}\chi^2\| \nabla \phi \|_{L^{\infty}_{t,x}}^2\| s_0 \|_{L^1_x}^{2}\leq \frac18$, then we can have
\bq\label{May-2nd-2}
\int_{\R^3}  s \log s dx+\frac14\int_0^t\int_{\R^3} | \nabla \sqrt{s}|^2 dxd\tau < Ct.
\eq
{Let  $(\log s)_{-}$ be a negative part of $\log s$ and $\langle x \rangle=(1+|x|^2)^{\frac12}$. Decomposing the domain $\{ x | s(x) \le 1\} $ into
$ D_1\cup D_2 = : \{ x| 0 \le s(x) \le e^{-|x|} \} \cup \{ x|   e^{-|x|}  \le s(x) \le 1 \}$
and using $x (\log x)_- < C\sqrt x$ for the integral over $D_1$, we have}
\bq\label{May-2nd-1}
\int_{\R^3} s (\log s)_{-} \leq C\int_{\R^3} e^ { - |x|} + C \int_{\R^3} \langle x\rangle s.
\eq

Integration by parts gives us that
\[
\frac{d}{dt} \int_{\R^3} \langle  x\rangle s = \int_{\R^3} s (u \cdot \nabla) \langle x \rangle + \int_{\R^3} s \Delta \langle x\rangle +\int_{\R^3} \chi n \nabla c \cdot \nabla \langle x \rangle -\epsilon \int_{\R^3} \langle x \rangle se.
\]
Since $|\nabla \langle x \rangle | +|\Delta \langle x \rangle | \leq C $, we have
\[
\left| \int_{\R^3} s (u\cdot \nabla) \langle x \rangle     \right|\leq C \| \sqrt{s} \|_{L^\frac{12}{5}}^2 \| u \|_{L^6} \leq C \| \sqrt{s} \|_{L^2}^{\frac32}\| \nabla \sqrt{s} \|_{L^2}^{\frac12}\| u \|_{L^6}
\]
\[
\leq \delta \| \nabla \sqrt{s} \|_{L^2}^2 +C \| s_0 \|_{L^1} \| \nabla u\|_{L^2}^2 
\]
 and
\[
\left| \int_{\R^3} s \Delta \langle x \rangle  \right|+ \left| \int_{\R^3} \chi s \nabla c \cdot \nabla \langle x \rangle  \right| \leq C+C \| \nabla \sqrt{s} \|_{L^2}^{\frac12} \| \nabla c \|_{L^6}
\]
\[
\leq C+ \delta \| \nabla \sqrt{s} \|_{L^2}^2+C \| \nabla c \|_{L^6}^2,
\]
for sufficiently small $\delta >0$.\\
Also from the equation $\partial_t c -\Delta c =-\nabla \cdot (uc) +e$, we have
\[
\| \nabla c\|_{L^{2,6}_{t,x}}^2 \leq C(\| uc \|_{L^{2,6}_{t,x}}^2 +1) \leq C (\| u \|_{L^{2,6}_{t,x}}^2 +1) \leq C(\| \nabla u \|_{L^{2,2}_{t,x}}^2+1).
\]
Considering \eqref{May-2nd-1} and adding $2 \int_{\R^3} s(\log s)_{-}$ on the both sides of \eqref{May-2nd-2}, we obtain
\bq \label{May-2nd-3}
\int_{\R^3}  s(t) |\log s(t)| dx+\frac18\int_0^t\int_{\R^3} | \nabla \sqrt{s}|^2 dxd\tau < C(t+1)+ C_{**} \int_0^t \| \nabla u \|_{L^2}^2.
\eq
From the equation of $u$, we deduce that
\[
\frac12 \frac{d}{dt} \| u \|_{L^2}^2 +\| \nabla u \|_{L^2}^2 \leq C(\| s \|_{L^{\frac65}}+\| e \|_{L^{\frac65}}) \| u \|_{L^6}^2
\]
\[
\leq C +\delta \| \nabla \sqrt{s} \|_{L^2}^2 +\delta \| \nabla u \|_{L^2}^2.
\]
Multiplying $4C_{**}$ on the both sides of the above inequality and integrating with respect to time, we have
\bq\label{May-2nd-4}
2C_{**}\int_{\R^3} |u|^2 (t) dx + 2C_{**}\int_0^t \int_{\R^3} |\nabla u |^2 \leq Ct + 4C_{**}\delta \int_0^t \| \nabla \sqrt{s} \|_{L^2}^2 d\tau.
\eq
If we add \eqref{May-2nd-3} and \eqref{May-2nd-4}, then we have
\[
\int_{\R^3}  s(t) |\log s(t)| dx+2C_{**}\int_{\R^3} |u|^2 (t) dx
\]
\bq\label{May-2nd-5}
+\frac{1}{16}\int_0^t\int_{\R^3} | \nabla \sqrt{s}|^2 +C_{**}\int_0^t \int_{\R^3} |\nabla u |^2 \leq C(1+t).
\eq
Hence we have
\[
\nabla \sqrt{s} \in L^2(0, t; L^2(\R^3))\mbox{ i.e., } s \in L^1(0, t; L^3(\R^3)). 
\]
From the interpolation, it gives us that
\bq \label{3ds}
s \in L^{q,p}_{t,x} \quad\mbox{ with }\quad\frac{3}{p}+\frac{2}{q}=3,\,\, 1 \leq p \leq 3.
\eq
By the maximal regularity estimate for Stokes equation \eqref{stokesmax}, we obtain 
\[
\| \Delta u \|_{L^{ 5,\frac{15}{13}}_{t,x}}^5 \leq C(\| u_0 \|_{W^{2, \frac{15}{13}}}+ \| s \|_{L^{5,\frac{15}{13}}_{t,x}}^5)< C(\| u_0 \|_{W^{2, \frac{15}{13}}}+ \| s_0 \|_{L^1}^4 + \delta \int_0^t \| \nabla \sqrt{s} \|_{L^2}^2),
\]
and hence  $u\in L^{5}_{t,x}$ by the Sobolev embedding. 
%By the estimates of the Stokes' equations, we obtain
%\[
%\| \Delta u \|_{L^{\frac{5}{4}, \frac{10}{3}}_{x,t}} \leq C(1+ \| s \|_{L^{\frac{5}{4}, \frac{10}{3}}_{x,t}})< \infty.
%\]
%Then we have  $\| u \|_{L^{\frac{15}{2}, \frac{10}{3}}_{x,t}} <\infty$. \\
%Since we have $s \in L^{1,\infty}_{x,t} \cap L^{3,1}_{x,t}$, we have $\Delta u \in L^{\frac{15}{13},5}_{x,t} $ . 
Also we have
\[ \frac12 \frac{d}{dt} \| \nabla c \|_{L^2}^2 + \| \Delta c \|_{L^2}^2 = \int_{\R^3} u \nabla c \Delta c + \int_{\R^3} e \Delta c \]
\[
\leq \| u \|_{L^5} \| \nabla c \|_{L^{\frac{10}{3}}} \| \Delta c \|_{L^2}+ \frac18 \| \Delta c \|_{L^2}^2 +C
\]
\[
\leq C\| u \|_{L^5} \| \nabla c \|_{L^{2}}^{\frac25} \| \Delta c \|_{L^2}^{\frac85}+ \frac18 \| \Delta c \|_{L^2}^2 +C
\]
\[
\leq  C\| u \|_{L^5}^5 \| \nabla c \|_{L^{2}}^2 +\frac14 \| \Delta c \|_{L^2}^2 +C.
\]
Hence we have $\nabla c \in L^{\infty,2}_{t,x}$ and $\Delta c \in L^{2}_{t,x}$.\\

Also from the equation $\partial_t c -\Delta c =-\nabla \cdot (uc) +e$, we have
\[
\| \nabla c\|_{L^5_{x,t}} \leq C(\| uc \|_{L^5_{x,t}} +1) \leq C (\| u \|_{L^{5}_{x,t}} +1) < \infty.
\]
Hence we have
\[
\frac12 \frac{d}{dt} \int_{\R^3} |s|^2 +\int_{\R^3} |\nabla s |^2 \leq C\left|   \int_{\R^3} s\nabla c \cdot \nabla s\right|
\]
\[
\leq C \|s \|_{L^{\frac{10}{3}} } \| \nabla c\|_{L^5} \| \nabla s \|_{L^2}
\]
\[
\leq \| \nabla s \|_{L^2}^{\frac85} \| s \|_{L^2}^{\frac25} \| \nabla c\|_{L^5} \leq \frac12 \| \nabla s \|_{L^2}^2 +C \| \nabla c\|_{L^5}^5 \| s \|_{L^2}^2.
\]
By using Gronwall's inequality, we have $ s\in L^{ \infty,2}_{t,x}$ and $\nabla s \in L^{2}_{t,x}$. The higher order estimates can be obtained in a similar fashion. A brief sketch of the proof is as follows : as in \cite[Theorem 1]{CKL2}, we can show that 
\[
\| \nabla u \|_{L^{5}_{t,x} } \leq C( \| s \|_{L^{5, \frac{15}{8}}_{t,x}}+1)
\]
and
\[
\| \nabla^2 c \|_{L^{\infty, 2}_{t,x}} + \| \nabla^3 c \|_{L^2_{t,x} } \leq C (\| \nabla u \|_{L^{5}_{t,x}}+1).
\]
Also we can show that if $T^{*}$ is a finite maximal existence time, then 
\[
\| \nabla c \|_{L^{2, \infty}_{t,x}(Q_{T^*})} = \infty.
\]
But, by the previous estimates and the standard continuation argument, we can complete the proof of existence of solution in (ii). For the positive lower bound of the total mass, we can obtain the lower bounds in (ii). The proof of the last part in (ii) is parallel to the proof of Theorem \ref{thm2} (ii) and we omit the details.
\\

\section{Decay estimates in Theorem $2$}
In this section we prove the part \textbf{(iii)} of Theorem \ref{thm-exist-es-NS}. 
 From the equation of $e$, we have
\[
\frac{d}{dt}\int_{\R^d} |e(t) |^p dx +C \| e\|_{L^p(\R^d)}^{\frac{p(d+2)}{d}} \| e \|_{L^{\frac{p}{2}}(\R^d)}^{-\frac{2p}{d}} \leq 0.
\]
Following the same proof for Theorem\ref{thm2} (ii) (see \eqref{eq-May-25-1} below),  we have \eqref{decay-e-22}
\[
\| e(t) \|_{L^p} \lesssim \frac{1}{t^{\left( 1-\frac{1}{p}\right)\frac{d}{2}}},\quad  1<p \leq \infty.
\]

Next  we will obtain the decay estimate of $\|c\|_{L^q}$ when $d=3$.

Noting first that $\norm{c(t)}_{L^1}\le Ct$ for sufficiently large $t$, we have
\[
\frac13 \frac{d}{dt} \int_{\R^3} |c|^3 +t^{-1} \left(\int_{\R^3} |c|^3\right)^{\frac43} \lesssim\frac13 \frac{d}{dt} \int_{\R^3} |c|^3 +\frac49 \int_{\R^3}\left|\nabla c^{\frac32} \right|^2 \lesssim t^{-1}\| c\|_{L^3}^2,
\] 
where we used that $\norm{c}_{L^3}\le \norm{c}^{\frac{1}{4}}_{L^1}\norm{c}^{\frac{3}{4}}_{L^9}$.
Setting $x(t)=\norm{c(t)}_{L^3}$, the above inequality can be rewritten as $x'(t)\le t^{-1}(C_1-C_2x^2(t))$ for some constants $C_1$ and $C_2$. Since it is a separable form of 1st order ordinary differential inequality,
direct computations show that $\| c(t)\|_{L^3}$ is uniformly bounded. Its verification is rather standard, and thus we skip its details.

 We next show that
$\norm{c(t)}_{L^q}$ is uniformly bounded for  $3<q\le\infty$, $d=3$. Let $3<q<\infty$. Using
$\norm{e(t)}_{L^q}\lesssim t^{-\frac{3}{2}(1-\frac{1}{q})}$, we have
\[
\frac{1}{q}\frac{d}{dt}\int_{\R^3}\abs{c}^q+\frac{4}{q^2}\int_{\R^3}\abs{\nabla
c^{\frac{q}{2}}}^2=\int_{\R^3}ec^{q-1}\le
\norm{e}_{L^q}\norm{c}^{q-1}_{L^q}\lesssim
t^{-\frac{3}{2}(1-\frac{1}{q})}\norm{c}^{q-1}_{L^q}.
\]
Noting that
\[
\norm{c}_{L^q}\le
\norm{c}^{\frac{2}{q-1}}_{L^{3}}\norm{c}^{\frac{q-3}{q-1}}_{L^{3q}}\lesssim
\norm{c}^{\frac{q-3}{q-1}}_{L^{3q}}\lesssim
\norm{c^{\frac{q}{2}}}^{\frac{2(q-3)}{q(q-1)}}_{L^{6}}\lesssim
\norm{\nabla c^{\frac{q}{2}}}^{\frac{2(q-3)}{q(q-1)}}_{L^{2}},
\]
we see that
\[
\frac{1}{q}\frac{d}{dt}\int_{\R^3}\abs{c}^q+\frac{4}{q^2}\bke{\int_{\R^3}\abs{c}^q}^{\frac{q-1}{q-3}}\lesssim
t^{-\frac{3}{2}(1-\frac{1}{q})}\norm{c}^{q-1}_{L^q},
\]
which can be rewritten as, denoting $y(t):=\norm{c(t)}_{L^q}$,
\begin{equation}\label{Jan-13-30}
y'(t)+\frac{1}{q^2}(y(t))^{\frac{3(q-1)}{q-3}}\lesssim
t^{-\frac{3}{2}(1-\frac{1}{q})}.
\end{equation}
By solving the differential inequality \eqref{Jan-13-30}, we can deduce  \eqref{Jan-13-40}
\begin{equation*}
\norm{c(t)}_{L^q}\lesssim
t^{-\frac{3}{2}(\frac{1}{3}-\frac{1}{q})},\qquad q>3.
\end{equation*}

Next, we prove \eqref{eq-May-25-2} for $d=2$. 
{First we remind the linear heat kernel estimates in $\R^2$;
\begin{align}\label{kernel}\begin{aligned}
\| \na ^{\al} S(t)f \|_{L^q}&  \le
C t^{- (1/r-1/q) - |\al|/2} \| f \|_{L^r}, \qquad 1 \le r\le
q \le \infty, \\
\| \nabla S(t) f \|_{L^q} &\leq Ct^{-\left(\frac12-\frac{1}{q}\right)} \| \nabla f \|_{L^2}, \quad \quad 2\le q\le \infty,
\end{aligned}\end{align}
and
\begin{align}\label{inho}
\int_0^t \| \nabla S(t-\tau) f(\tau) \|_{L^q} d\tau &\le C \int_0^t \frac{1}{(t-\tau)^{\frac32-\frac{1}{\alpha_0}}}\cdot \frac{1}{\tau^{1-\frac{1}{l}}}d\tau \| f\|_{{\mathcal{K}}_{l}}
\le C \frac{1}{t^{\frac12-\frac{1}{q}}}\| f\|_{{\mathcal{K}}_{l}}
\end{align}
with $1+\frac{1}{q}=\frac{1}{\alpha_0}+\frac{1}{l}$. 
Also we use the following elementary results on the integral for any $a>0,\, b>0$ and $0<a,\, b<1$
\[
\int_0^t \frac{1}{(t-s)^{1-a}} \frac{1}{s^{1-b}} ds \le
\frac{C}{t^{1-(a+b)}} ,\qquad (a>0, \, b>0),
\]
\[
\int_0^{\frac t2} \frac{1}{(t-s)^b} \frac{1}{s^{1-a}} ds  \le
\frac{C}{t^{b-a}}, \quad 
\int_{\frac t2}^t \frac{1}{(t-s)^{1-a}}
\frac{1}{s^{b}} ds  \le \frac{C}{t^{b-a}} \qquad (a>0,  b\ge 0).
\]
}\\
\indent
Let us introduce the function spaces used in \cite{CKL1}.
\begin{align}\label{lowindex}
\begin{aligned}
\norm{c}_{{\mathcal{N}}_q}& :=\sup_{t} t^{\frac{1}{2}-\frac{1}{q}}\norm{\nabla
c}_{L^q(\R^2)},\qquad 2<q<4, \\
\norm{s}_{{\mathcal{K}}_p}&:=\sup_{t}
t^{1-\frac{1}{p}}\norm{s}_{L^p(\R^2)},\qquad \frac{4}{3}<p<2,\\
\norm{e}_{{\mathcal{K}}_l}&:=\sup_{t}
t^{1-\frac{1}{l}}\norm{e}_{L^l(\R^2)},\qquad 1<l\le \infty,\\
\norm{\omega}_{{\mathcal{K}}_{\gamma}}&:=\sup_{t}
t^{1-\frac{1}{\gamma}}\norm{\omega}_{L^{\gamma}(\R^2)},\qquad
1<\gamma<2.
\end{aligned}
\end{align}
From \eqref{decay-e-22}, we already obtain $\norm{e}_{{\mathcal{K}}_l} \leq C(\epsilon_1)$.

Let $\Gamma(x,t) $ be the two dimensional heat kernel, i.e., 
\[
\Gamma(x,t) =(4 \pi t)^{-1}\exp \left(-|x|^2/4t  \right).
\]
If we set $$S(t) u = \int_{\R^2} \Gamma(x-y, t)u(y) dy,$$ then we write the equations as the integral representation.
\[
s(t)=S(t) s_0 -  \int_0^t \nabla S(t-\tau)\cdot\left[ s(\tau) \nabla c(\tau) + u(\tau) s(\tau) \right]\, d\tau -\epsilon \int_0^t S(t-\tau) (s(\tau) \, e(\tau))\, d\tau,
\]
\[
e(\tau) = S(t) e_0 - \int_0^t S(t-\tau) \left[u (\tau) \cdot \nabla e(\tau)+\epsilon s(\tau) e(\tau) \right] d \tau,
\]
\[
c(t) =S(t) c_0 -\int_0^t S(t-\tau) ( u(\tau) \cdot \nabla c(\tau) - e(\tau))\, d\tau,
\]
and
\[
\omega(t) = G(t) \omega_0 +\int_0^t \nabla G(t-\tau)\cdot \left[ (s(\tau)+e(\tau))\nabla^{\perp} \phi- u(\tau)\omega(\tau)     \right] \, d\tau.
\]
Using the estimate of the heat kernel, we obtain 
\[
\| s(t) \|_{L^p} \lesssim t^{-1+\frac{1}{p}} \| s_0 \|_{L^{1}} + \chi \int_0^t \| \nabla S(t-\tau) ( s(\tau) \nabla c(\tau)) \|_{L^p} d\tau
\]
\[
+\int_0^t \| \nabla S(t-\tau) (u(\tau) s(\tau)) \|_{L^p} d\tau +\epsilon \int_0^t \| S(t-\tau) (s(\tau) e(\tau) ) \|_{L^p} d\tau
\]
\[
\lesssim t^{-1+\frac{1}{p}} \| s_0 \|_{L^1} +\chi \int_0^t \frac{1}{(t-\tau)^{\frac{3}{2}-\frac{1}{\alpha}}} \| s(\tau) \|_{L^p} \| \nabla c (\tau) \|_{L^q} d\tau
\]
\[
+\int_0^t \frac{1}{(t-\tau)^{\frac{3}{2}-\frac{1}{\alpha'}}}\| u(\tau)\|_{L^{\frac{2r}{2-r}}} \| s(\tau) \|_{L^p} d\tau + \epsilon\int_0^t \frac{1}{(t-\tau)^{\left(1-\frac{1}{\beta}\right)}} \| s(\tau) \|_{L^p}\|e(\tau) \|_{L^l} d\tau
\]
\[
:=  t^{-1+\frac{1}{p}} \| s_0 \|_{L^1}+I_1+I_2+I_3,
\]
where $1+\frac{1}{p}=\frac{1}{\alpha} +\frac{1}{p} +\frac{1}{q}$, $1+\frac{1}{2}-\frac{1}{r}=\frac{1}{\alpha'}$,
and $1+\frac{1}{p}=\frac{1}{\beta}+\frac{1}{p}+\frac{1}{l}$.
 We estimate $I_1$, $I_2$ and $I_3$ as follows:
\[
I_1 \lesssim \chi\int_0^t \frac{1}{(t-\tau)^{\frac32-\frac{1}{\alpha}}} \cdot \frac{1}{\tau^{\frac32-\frac{1}{p}-\frac{1}{q}}} d\tau \| s \|_{{\mathcal{K}}_{p}} \| c \|_{{\mathcal{N}}_{q} } \lesssim
\frac{\chi}{t^{1-\frac{1}{p}}} \|s \|_{{\mathcal{K}}_{p}} \| c \|_{{\mathcal{N}}_{q} }, \]
\[
I_2 \lesssim \int_0^t \frac{1}{(t-\tau)^{\frac32-\frac{1}{\alpha'}} }\cdot \frac{1}{\tau^{2-\frac{1}{r}-\frac{1}{p}}} d\tau \| \omega \|_{{\mathcal{K}}_r} \| s \|_{{\mathcal{K}}_p} \lesssim \frac{1}{t^{1-\frac{1}{p}}} \| \omega \|_{{\mathcal{K}}_r}  \| s \|_{{\mathcal{K}}_p},
\]
and
\[
I_{3} \lesssim \epsilon\int_0^t \frac{1}{(t-\tau)^{1-\frac{1}{\beta}}} \cdot \frac{1}{\tau^{2-\frac{1}{p}-\frac{1}{l}}} d\tau \| s\|_{{\mathcal{K}}_p} \| e \|_{{\mathcal{K}}_{l}}\lesssim \frac{\epsilon}{t^{1-\frac{1}{p}}} \| s\|_{{\mathcal{K}}_p} \| e \|_{{\mathcal{K}}_{l}},
\]
{where we use the embedding $\| u\|_{L^{\frac{2r}{2-r}} }\lesssim \| \omega\|_{L^r}$, hence $1 <r <2$ is required.}
Therefore, we 
deduce that{ for any exponent $ p, q, r, l$ in \eqref{lowindex}}
\bq\label{180811-1}
\| s \|_{{\mathcal{K}}_{p}} \leq C\| s_0 \|_{L^1} +C \| s\|_{{\mathcal{K}}_{p}}(\chi\| c\|_{{\mathcal{N}}_q} + \| \omega \|_{{\mathcal{K}}_r} + \epsilon \| e \|_{{\mathcal{K}}_{l}}).
\eq
Similarly, we obtain
\bq\label{180811-2}
\| e \|_{{\mathcal{K}}_{l}} \leq C\| e_0 \|_{L^1} +C \| e\|_{{\mathcal{K}}_{l}}( \| \omega \|_{{\mathcal{K}}_r} + \epsilon \| s \|_{{\mathcal{K}}_{p}}).
\eq
By applying \eqref{kernel}, \eqref{inho} to the $c$ equation
%\[
%\| \nabla S(t) c_0 \|_{L^q} \leq Ct^{-\frac{d}{2}\left(\frac12-\frac{1}{q}\right)} \| \nabla c_0 \|_{L^2},
%\]
%and 
%\[
%\int_0^t \| \nabla S(t-\tau) e(\tau) \|_{L^q} d\tau \lesssim \int_0^t \frac{1}{(t-\tau)^{\frac32-\frac{1}{\alpha_0}}}\cdot \frac{1}{\tau^{1-\frac{1}{l}}}d\tau \| e\|_{{\mathcal{K}}_{l}}\lesssim \frac{1}{t^{\frac12-\frac{1}{q}}}\| e\|_{{\mathcal{K}}_{l}},
%\]
%with $1+\frac{1}{q}=\frac{1}{\alpha_0}+\frac{1}{l}$, 
we  easily deduce that
\bq\label{180811-3}
\| c\|_{{\mathcal{N}}_{q}} \leq C \| \nabla c_0 \|_{L^2} +C \| c \|_{{\mathcal{N}}_{q}}\| \omega \|_{{\mathcal{K}}_{r}} +C_{*} \| e\|_{{\mathcal{K}}_l}.
\eq
Next by similar computaions as in \cite[Lemma 3]{CKL1}, we obtain that
\bq\label{180811-4}
\| \omega \|_{{\mathcal{K}}_r} \leq C \| \omega_0 \|_{L^1} +C\| \nabla \phi \|_{L^2}( \| s \|_{{\mathcal{K}}_p}+  \|e \|_{{\mathcal{K}}_l}) +C \| \omega \|_{{\mathcal{K}}_r}^2.
\eq
Here we set $M_{1} :=C_{*}$ and $M_{2}=C\| \nabla \phi \|_{L^2}$, where $C_{*}$ and $C\| \nabla \phi\|_{L^2}$ are the constants in \eqref{180811-3} and \eqref{180811-4} respectively.  
{Indeed,
\[
\norm{\omega (t)}_{L^r} \lesssim t^{-1+\frac{1}{r}}
\norm{\omega_0}_{L^1}+\int_0^t \norm{\nabla
G(t-\tau)(s+e)(\tau)\nabla \phi }_{L^{r}}+\int_0^t \norm{\nabla
G(t-\tau) u(\tau)\omega(s)}_{L^r}
\]
\[
\lesssim t^{-1+\frac{1}{r}} \norm{\omega_0}_{L^1}+\int_0^t
\frac{1}{(t-\tau)^{\frac32-\frac{1}{\alpha}}}\norm{s(\tau)}_{L^p}
\norm{\nabla \phi }_{L^{2}}+\frac{1}{(t-\tau)^{\frac32-\frac{1}{\beta}}}\norm{e(\tau)}_{L^l}
\norm{\nabla \phi }_{L^{2}}
%+\int_0^t
%\frac{1}{(t-s)^{\frac32-\frac{1}{\alpha'}}}
%\norm{u}_{L^{\frac{2r}{2-r}}} \norm{\omega}_{L^r}
\]
\[
+\int_0^t \frac{1}{(t-\tau)^{\frac32-\frac{1}{\alpha'}}}
\norm{u}_{L^{\frac{2r}{2-r}}} \norm{\omega}_{L^r}
=t^{-1+\frac{1}{r}} \norm{\omega_0}_{L^1}+J_1+J_2+J_3,
\]
where $\frac{1}{r}=\frac{1}{\alpha} +\frac{1}{p} -\frac{1}{2}$, $\frac{1}{r}=\frac{1}{\beta} +\frac{1}{p} -\frac{1}{2}$ and
$\frac{1}{\alpha'}=\frac32-\frac{1}{r}$. Similar estimates as above
yield that
\[
J_1 \lesssim \int_0^t \frac{1}{(t-\tau)^{\frac32-\frac{1}{\alpha}}}
\frac{1}{\tau^{1-\frac{1}{p}}} ds \norm{\nabla \phi
}_{L^{2}}\norm{s}_{\mathcal{K}_p} \lesssim
\frac{1}{t^{1-\frac{1}{r}}}\norm{\nabla \phi
}_{L^{2}}\norm{s}_{\mathcal K_p}
\]
and
\[
J_2 \lesssim \frac{1}{t^{1-\frac{1}{r}}}\norm{\nabla \phi
}_{L^{2}}\norm{e}_{\mathcal K_l}.
\]
}
{On the other hand, via $\norm{u(t)}_{L^s}\lesssim
\norm{\omega(t)}_{L^r}$ with $1/r=1/s+1/2$, we obtain
% \begin{equation}\label{Nov28-CKL200}
%t^{\frac{1}{2}-\frac{1}{s}}\norm{u(t)}_{L^s}\lesssim
%t^{1-\frac{1}{r}}\norm{\omega(t)}_{L^r},\qquad
%\frac{1}{r}=\frac{1}{s}+\frac{1}{2}.
%\end{equation}
\[
J_3 \lesssim \int_0^t \frac{1}{(t-s)^{\frac32-\frac{1}{\alpha'}}}
\frac{1}{s^{2(1-\frac{1}{r})}} ds \norm{\omega }_{\mathcal K_r}^2
\lesssim\frac{1}{t^{1-\frac{1}{r}}}\norm{\omega }_{\mathcal K_r}^2.
\]
Thus, we have \eqref{180811-4}.
}
Multiplying \eqref{180811-1} and \eqref{180811-2} with $2M_2$ and $2(M_1+M_2)$ ($M_1$ and $M_2$ are large constants, which are larger than $\| \nabla \phi \|_{L^2}$ and $C_*$), respectively, and summing up above estimates, we have
\[
M_2 \| s\|_{{\mathcal{K}}_p} +(M_1+M_2) \| e \|_{{\mathcal{K}}_l}+\| c\|_{{\mathcal{N}}_q} +\| \omega \|_{{\mathcal{K}}_r}
\]
\[
\leq C ( \| s_0 \|_{L^1} + \| e_0 \|_{L^1}+\| \omega_0 \|_{L^1} +\| \nabla c_0 \|_{L^2}) + \left(\| s\|_{{\mathcal{K}}_p}+\| e\|_{{\mathcal{K}}_l}+ \|\omega \|_{{\mathcal{K}}_r }+ \| c\|_{{\mathcal{N}}_q}\right)^2.
\]
Under the smallness assumption, we have
\begin{align}\label{lowrange}
\norm{(s, e, \omega)}_{\mathcal{K}_{p,l,r}}+ \| c\|_{{\mathcal{N}}_q}\lesssim
\norm{s_0}_{L^1}+\| e_0 \|_{L^1}+\norm{\nabla c_0}_{L^2}+\norm{\omega_0}_{L^1}\lesssim \epsilon_1.
\end{align}
{Now we extend the range of $p, r$ of $ \|s\|_{\calK_p}$, $\|\omega\|_{\calK_r}$ and consider $\|c\|_{\calN_{\infty}}$ such that}
\begin{align*}
\norm{ s }_{{\mathcal{K}}_{p}} &:= \sup_{t \geq 0} t^{1-\frac{1}{p}} \| s(t) \|_{L^{p}}, \quad 2\leq p  \le \infty,\\
%\norm{ e }_{{\mathcal{K}}_{\infty}} &:= \sup_{t \geq 0} t \| e(t) \|_{L^{\infty}},\\
\norm{c }_{{\mathcal{N}}_{\infty}}& := \sup_{t \geq 0} t^{\frac12} \| \nabla c(t) \|_{L^{\infty}}, \\
\norm{ \omega }_{{\mathcal{K}}_{r}} &:= \sup_{t \geq 0} t^{1-\frac{1}{r}} \| \omega(t) \|_{L^{r}},\quad 1<r<\infty.
\end{align*}
%Note that we already have
%\[  
%\norm{ e }_{{\mathcal{K}}_{\infty}} \leq \epsilon_1.
%\]
Since $\int_0^t S(t-\tau)(se)(\tau) d\tau$ is always nonnegative, we have
\[
\| s \|_{L^\infty}(t) \lesssim  t^{-1} \| s_0 \|_{L^{1}} + \chi \int_0^t \| \nabla S(t-\tau) ( s(\tau) \nabla c(\tau)) \|_{L^\infty} d\tau
\]
\[
+\int_0^t \| \nabla S(t-\tau) (u(\tau) s(\tau)) \|_{L^\infty} d\tau 
:= t^{-1+\frac{1}{p}} \| s_0 \|_{L^{1}}+I_1+I_2.
\]
$I_1$ and $I_2$ can be estimated as follows:
\[
I_1 \lesssim \int_0^{\frac{t}{2}}\frac{1}{(t-\tau)^{\frac32}} \| s \nabla c\|_{L^1} (\tau) d\tau+\int_{\frac{t}{2}}^{t} \frac{1}{(t-\tau)^{\frac12}}\|s \nabla c \|_{L^\infty}(\tau) d\tau
\]
\[
\lesssim \int_0^{\frac{t}{2}} \frac{1}{(t-\tau)^{\frac32}}\| s\|_{L^1} \| \nabla c \|_{L^{\infty}} d\tau +\int_{\frac{t}{2}}^t \frac{1}{(t-\tau)^{\frac12}} \| s\|_{L^\infty} \| \nabla c\|_{L^{\infty}} d\tau
\]
\[
\lesssim \frac{\epsilon_1}{t^{1}} \|  c \|_{{\mathcal{N}}_{\infty}}+\frac{1}{t^{1}} \| s \|_{{\mathcal{K}}_\infty} \| c \|_{{\mathcal{N}}_{\infty}},
\]
and
\[
I_2 \lesssim \int_0^{\frac{t}{2}} \frac{1}{(t-\tau)^{\frac32}} \| us \|_{L^1} d\tau +\int_{\frac{t}{2}}^{t} \frac{1}{(t-\tau)^{\frac32-\frac{1}{2^{-}}}} \| us \|_{L^{2^{+}}} d\tau
\]
\[
\lesssim \int_0^{\frac{t}{2}} \frac{1}{(t-\tau)^{\frac32}} \| u \|_{L^{2^{+}}} \| s \|_{L^{2^{-}}} + \int_{\frac{t}{2}}^{t} \frac{1}{(t-\tau)^{\frac32-\frac{1}{2^{-}}}} \| u \|_{L^{2^{+}}} \| s \|_{L^{\infty}}
\]
\[
\lesssim \frac{1}{t^{\frac32}} \int_0^{\frac{t}{2}}\| \omega \|_{L^{\frac{\alpha}{\alpha-1}}} \| s \|_{L^{2^{-}}} +\int_{\frac{t}{2}}^{t} \frac{1}{(t-\tau)^{\frac32-\frac{1}{2^{-}}}}\| \omega \|_{L^{\frac{\alpha}{\alpha-1}}}  \| s \|_{L^{\infty}}
\]
\[
\lesssim \frac{1}{t} \| s \|_{\mathcal{K}_{2^{-}}} \| \omega \|_{\mathcal{K}_{\frac{\alpha}{\alpha-1}}} + \frac{1}{t} \| s\|_{\mathcal{K}_{\infty}}\| \omega \|_{\mathcal{K}_{\frac{\alpha}{\alpha-1}}}
\lesssim \frac{\epsilon_1^2}{t}+\frac{\epsilon_1}{t}\| s\|_{\mathcal{K}_{\infty}}  ,
\]
where $\alpha$ and $p$ satisfy $2<\alpha<p$ and $2^{+}$ and $2^{-}$ satisfy $\frac{1}{2^{+}}=\frac12-\frac{1}{\alpha}+\frac{1}{p},$  and $\frac{1}{2^{-}}=\frac12+\frac{1}{\alpha}.$

Adding these estimates, we obtain that
\[
\|s \|_{{\mathcal{K}}_{\infty}} \lesssim \| s_0 \|_{L^1}  +\epsilon_1 \|  c \|_{{\mathcal{N}}_{\infty}} + \| s\|_{{\mathcal{K}}_p}( \| c\|_{{\mathcal{N}}_{\infty}}  +\epsilon_1)+\epsilon_1^2.
\]
Using the similar methods with above and the estimates in \cite{CKL1}, we have
\[
\|  c \|_{{\mathcal{N}}_{\infty}} \lesssim \| e_0 \|_{L^1}  + \epsilon_1^2+\| e\|_{{\mathcal{K}}_{\infty}} +\epsilon_1  \|  c \|_{{\mathcal{N}}_{\infty}}.
\]
{Indeed,
\[
\norm{\nabla c}_{L^{\infty}}(t)\lesssim
\frac{1}{t^{\frac{1}{2}}}\norm{c_0}_{L^{\infty}}
+\int_0^t\norm{\nabla S(t-\tau) e}_{L^{\infty}}(\tau)d\tau
\]
\[
+\int_0^t\norm{\nabla S(t-\tau)(u\nabla c)}_{L^{\infty}}(\tau)d\tau
=\frac{1}{t^{\frac{1}{2}}}\norm{c_0}_{L^{\infty}}+J_1+J_2.
\]
Firstly, we estimate $J_1$.
\begin{align}\label{j1c}\begin{aligned}
J_1&\lesssim \int_0^{t/2}\frac{1}{(t-\tau)^{\frac{3}{2}}}\norm{e(\tau)}_{L^1}ds+\int_{t/2}^t\frac{1}{(t-\tau)^{\frac{1}{2}}}\norm{e(\tau)}_{L^{\infty}}d\tau
\\
&
\lesssim
%\frac{\epsilon_1}{t^{\frac{1}{2}}}+
\frac{1}{t^{\frac{1}{2}}}\norm{e}_{L^1}
+\frac{1}{t^{\frac{1}{2}}}\norm{e}_{\calK_{\infty}(\R^2)}\lesssim
\frac{\epsilon_1}{t^{\frac{1}{2}}}
+\frac{1}{t^{\frac{1}{2}}}\norm{e}_{\calK_{\infty}(\R^2)}.
\end{aligned}\end{align}
%\begin{equation}\label{Dec03-CKL130}
%\lesssim \frac{\epsilon_1}{t^{\frac{1}{2}}}
%+\frac{\epsilon_1}{t^{\frac{1}{2}}}\norm{n}_{\calK_{\infty}(\R^2)}.
%\end{equation}
Before we estimate $J_2$, we set $1/4^+=1/4-1/\beta$ and
$1/4^-=1/4+1/\beta$ with $\beta>4$. We then estimate $J_2$.
\begin{align}\label{j2c}\begin{aligned}
J_2&\lesssim \int_0^{t/2}\frac{1}{t-\tau}\norm{u\nabla
c}_{L^2}d\tau+\int_{t/2}^t\frac{1}{(t-\tau)^{\frac{3}{2}-\frac{1}{2^{-}}}}
\norm{u\nabla c}_{L^{2^{+}}}(\tau)d\tau
\\
&\lesssim \frac{1}{t}\int_0^{t/2}\norm{u}_{L^{4^+}}\norm{\nabla
c}_{L^{4^-}}d\tau+\int_{t/2}^t\frac{1}{(t-\tau)^{\frac{3}{2}-\frac{1}{2^{-}}}}
\norm{u}_{L^{2^{+}}}\norm{\nabla c}_{L^{\infty}}(\tau)d\tau
\\
&
\lesssim
\frac{1}{t}\int_0^{t/2}\norm{\omega}_{L^{\frac{4\beta}{3\beta-4}}}\norm{\nabla
c}_{L^{4^-}}ds+\int_{t/2}^t\frac{1}{(t-s)^{\frac{3}{2}-\frac{1}{2^{-}}}}
\norm{\omega}_{L^{\frac{\alpha}{\alpha-1}}}\norm{\nabla
c}_{L^{\infty}}(\tau)d\tau
\\&
\lesssim
\frac{1}{t^{\frac{1}{2}}}\norm{\omega}_{\calK_{\frac{4\beta}{3\beta-4}}(\R^2)}
\norm{ c}_{\calN_{4^-}(\R^2)}+\frac{1}{t^{\frac{1}{2}}}
\norm{\omega}_{\calK_{\frac{\alpha}{\alpha-1}}(\R^2)} \norm{
c}_{\calN_{\infty}(\R^2)}\lesssim\frac{\epsilon^2_1}{t^{\frac{1}{2}}}
+\frac{\epsilon_1}{t^{\frac{1}{2}}}\norm{
c}_{\calN_{\infty}(\R^2)},
\end{aligned}\end{align}
where the estimates for the low range of $\| \omega\|_{\calK_p}$, $\| c\|_{\calN_q}$ \eqref{lowrange} is used
and $2+, 2-, \al$ are same exponents as for $I_2$ before.
 Combining
\eqref{j1c}and \eqref{j2c}, we have
\begin{equation*}
\norm{\nabla c}_{L^{\infty}}(t)\lesssim \frac{1}{t^{\frac{1}{2}}}\| c_0 \|_{L^{\infty}}+\frac{\epsilon_1^2}{t^{\frac{1}{2}}}
+\frac{1}{t^{\frac{1}{2}}}\norm{e}_{\calK_{\infty}(\R^2)}
+\frac{\epsilon_1}{t^{\frac{1}{2}}}\norm{
c}_{\calN_{\infty}(\R^2)}.
\end{equation*}
}
Next, we estimate the vorticity  for $2 \leq r< \infty$.
\[
\| \omega (t) \|_{L^r} \lesssim t^{-1+\frac{1}{r}}\| \omega_0 \|_{L^1} +\int_0^t \| \nabla ^{\perp} G(t-\tau) (s \nabla \phi)(\tau) \|_{L^r} d\tau
\]
\[
+\int_0^t \| \nabla ^{\perp} G(t-\tau) (e \nabla \phi)(\tau) \|_{L^r} d\tau+ \int_0^t \| \nabla  G(t-\tau) (u \omega)(\tau) \|_{L^r} d\tau 
\]
\[
= t^{-1+\frac{1}{r}}\| \omega_0 \|_{L^1}+K_1+K_2 +K_3.
\]
If we consider $r > 2$, then we obtain
\[
K_1 \lesssim \int_0^{t/2} \frac{1}{(t-\tau)^{\frac32-\frac{1}{r}}} \| s(\tau)\|_{L^1}^{\frac12} \| s(\tau) \|_{L^{\infty}}^{\frac12} \| \nabla \phi \|_{L^2} +\int_{t/2}^{t} \frac{1}{(t-\tau)^{1-\frac{1}{r}}} \| s(\tau) \|_{L^{\infty}} \| \nabla \phi \|_{L^2}
\]
\[
\lesssim \frac{\epsilon_1}{t^{1-\frac{1}{r}}} +\frac{1}{t^{1-\frac{1}{r}}}\| s \|_{\mathcal{K}_{\infty}}.
\]
Similarly, we have
\[
K_2 \lesssim \frac{\epsilon_1}{t^{1-\frac{1}{r}}} +\frac{1}{t^{1-\frac{1}{r}}}\| s \|_{\mathcal{K}_{\infty}}.
\]
If the exponents $r^{*}$, $\tilde{r}$ are defined by $\frac{1 }{r^{*}}= \frac12 -\frac{1}{r}$ and $\frac{1}{r^{*}}=\frac{1}{\tilde{r}}-\frac12$, then we estimate
\[
K_3 \lesssim \int_0^{\frac{t}{2}} \frac{1}{(t-\tau)^{\frac32- \frac{1}{2^{-}}}}\| u \|_{L^{2^{+}}}\| \omega \|_{L^r} + \int_{\frac{t}{2}}^{t} \frac{1}{(t-\tau)^{1-\frac{1}{r}}} \| u \|_{L^{r^*}}\| \omega \|_{L^r} 
\]
\[
\lesssim \frac{1}{t^{1-\frac{1}{r}}} \| \omega \|_{\mathcal{K}_{\frac{\alpha}{\alpha-1}}}\| \omega \|_{\mathcal{K}_r}+\frac{1}{t^{1-\frac{1}{r}}} \| \omega \|_{\mathcal{K}_{\tilde{r}}}\| \omega \|_{\mathcal{K}_r} \lesssim \frac{\epsilon_1}{t^{1-\frac{1}{r}}} \| \omega \|_{\mathcal{K}_r}.
\]
Thus, we have
\[
\| \omega \|_{\mathcal{K}_r} \lesssim \epsilon_1 + \| s \|_{\mathcal{K}_{\infty}} + \| e \|_{\mathcal{K}_{\infty}} +\epsilon_1 \| \omega \|_{\mathcal{K}_{r}}.
\]
By collecting all the estimates in the above, we find that
\[
\norm{s}_{{\mathcal{K}}_\infty}+ \norm{e}_{{\mathcal{K}}_\infty}+\norm{c}_{{\mathcal{N}}_{\infty}} +\| \omega \|_{\mathcal{K}_{r}}\lesssim \epsilon_1.
\]
%\textcolor{blue}{Couldn't we have $p=\infty$ in $\| s\|_{L_p}$ before?}
%For estimates of $\| s\|_{{\mathcal{K}}_{\infty}}$,  we  use
%the comparison argument of the heat equation. We recall the equation of $s$
%\[
%\partial_t s+(u\cdot \nabla)s -\Delta s=-\chi\nabla\cdot (s\nabla
%c)-\epsilon (se),\quad s(x,0)=s_0(x).
%\]
%Now we introduce $\tilde{s}$, which is defined as
%\begin{equation}\label{Jan-13-10}
%\partial_t \tilde{s} -\Delta \tilde{s}=-(u\cdot \nabla)s-\chi\nabla\cdot (s\nabla
%c),\quad \tilde{s}(x,0)=s_0(x).
%\end{equation}
%Therefore, $\overline{s}:=\tilde{s}-s$ solves
%\begin{equation}\label{Jan-13-20}
%\partial_t \overline{s} -\Delta \overline{s}=\epsilon (se),\quad \tilde{s}(x,0)=0.
%\end{equation}
%Since the righthand side of \eqref{Jan-13-20} is non-negative, due
%to the comparison argument, $\tilde{s}\ge s$. Then, \textcolor{blue}{ some arguments here} we have
%\[
%\|s \|_{{\mathcal{K}}_{\infty}} \le\|\tilde{s}
%\|_{{\mathcal{K}}_{\infty}} \lesssim \| s_0 \|_{L^1} +\epsilon_1 \|
%c \|_{{\mathcal{N}}_{\infty}}  + \| s\|_{{\mathcal{K}}_p}( \|
%c\|_{{\mathcal{N}}_{\infty}} + \| \omega
%\|_{{\mathcal{K}}_{\frac{\alpha}{\alpha-1}}}).
%\]
%This estimate gives $L^{\infty}$ estimate of $s$.
This completes the proof.
\hfill\fbox{}\par\vspace{.2cm}

\section*{Acknowledgements}
M. Chae's work is partially supported by NRF-2015R1C1A2A01054919. Kyungkeun Kang's work is partially supported by NRF-2017R1A2B400648. Jihoon Lee's work is partially supported by NRF- 2016R1A2B3011647 and SSTF-BA1701-05.
%We define a iterating sequence $(s_k, e_k)$ as follows:
%$(s_1(x,t), e_1(x,t))=(s(x,0), e(x,0))$ and  for $k \geq 1$,\\
%\begin{equation}
%\left\{
%\begin{array}{ll}
%\partial_t s_{k+1}-\Delta s_{k+1} = \chi \nabla \cdot (s_{k+1} \nabla \Delta^{-1} e_k)- \epsilon s_{k+1} e_{k},\\
%\partial_t e_{k+1} -\Delta e_{k+1} =-\epsilon s_{k} e_{k+1},\\
%(s_k(x,0), e_k(x,0))=(s(x,0), e(x,0)).
%\end{array}
%\right.
%\end{equation}
%Similarly with \cite{Kise-Ryz1}, we obtain for some $T_1>0$ and $m\geq 3$,
%\[
%\| (s_k,\, e_k)\|_{X^{T_1}_{m,n}} \leq M\mbox{ for } k \geq 1.
%\]
%To show $(s_k,\, e_k) \rightarrow (s,\, e)$

\begin{equation*}
\left.
\begin{array}{cc}
{\mbox{Myeongju Chae}}\qquad&\qquad {\mbox{Kyungkeun Kang}}\\
{\mbox{Department of Applied Mathematics }}\qquad&\qquad
 {\mbox{Department of Mathematics}} \\
{\mbox{Hankyong National University
}}\qquad&\qquad{\mbox{Yonsei University}}\\
{\mbox{Ansung, Republic of Korea}}\qquad&\qquad{\mbox{Seoul, Republic of Korea}}\\
{\mbox{mchae@hknu.ac.kr }}\qquad&\qquad {\mbox{kkang@yonsei.ac.kr }}
\end{array}\right.
\end{equation*}
\begin{equation*}
\left.
\begin{array}{cc}
{\mbox{Jihoon Lee}}\\
{\mbox{Department of Mathematics}} \\
{\mbox{Chung-Ang University}}\\
{\mbox{Seoul, Republic of Korea}}\\
{\mbox{jhleepde@cau.ac.kr }}
\end{array}\right.
\end{equation*}

\end{document}